\documentclass[11pt]{amsart}
\usepackage{}
\usepackage{amssymb}
\usepackage{amsfonts}
\usepackage{mathrsfs}
\usepackage{latexsym}
\usepackage{graphicx}
\usepackage{amscd,amssymb,amsmath,amsbsy,amsthm,amsfonts}
\usepackage[all]{xy}
\usepackage[colorlinks,plainpages,backref,urlcolor=blue]{hyperref}
\usepackage{verbatim}
\usepackage{enumerate}

\usepackage [english]{babel}
\usepackage [autostyle, english = american]{csquotes}
\MakeOuterQuote{"}

\newcommand {\Omit}[1]{}

\usepackage{tikz-cd}

\usepackage{tikz}
\usetikzlibrary{automata,positioning}

\usetikzlibrary{arrows,calc}
\tikzset{
>=stealth',
help lines/.style={dashed, thick},
axis/.style={<->},
important line/.style={thick},
connection/.style={thick, dotted},
}
\usetikzlibrary{patterns}
\newlength{\hatchspread}
\newlength{\hatchthickness}
\tikzset{hatchspread/.code={\setlength{\hatchspread}{#1}},
         hatchthickness/.code={\setlength{\hatchthickness}{#1}}}
\tikzset{hatchspread=3pt,
         hatchthickness=0.4pt}
\pgfdeclarepatternformonly[\hatchspread,\hatchthickness]
   {custom north west lines}
   {\pgfqpoint{-2\hatchthickness}{-2\hatchthickness}}
   {\pgfqpoint{\dimexpr\hatchspread+2\hatchthickness}{\dimexpr\hatchspread+2\hatchthickness}}
   {\pgfqpoint{\hatchspread}{\hatchspread}}
   {
    \pgfsetlinewidth{\hatchthickness}
    \pgfpathmoveto{\pgfqpoint{0pt}{\hatchspread}}
    \pgfpathlineto{\pgfqpoint{\dimexpr\hatchspread+0.15pt}{-0.15pt}}
    \pgfusepath{stroke}
   }

\newcommand{\nc}{\newcommand}
\nc{\rnc}{\renewcommand}
\nc{\bb}[1]{{\mathbb #1}}
\nc{\bbA}{\bb{A}}\nc{\bbB}{\bb{B}}\nc{\bbC}{\bb{C}}\nc{\bbD}{\bb{D}}
\nc{\bbE}{\bb{E}}\nc{\bbF}{\bb{F}}\nc{\bbG}{\bb{G}}\nc{\bbH}{\bb{H}}
\nc{\bbI}{\bb{I}}\nc{\bbJ}{\bb{J}}\nc{\bbK}{\bb{K}}\nc{\bbL}{\bb{L}}
\nc{\bbM}{\bb{M}}\nc{\bbN}{\bb{N}}\nc{\bbO}{\bb{O}}\nc{\bbP}{\bb{P}}
\nc{\bbQ}{\bb{Q}}\nc{\bbR}{\bb{R}}\nc{\bbS}{\bb{S}}\nc{\bbT}{\bb{T}}
\nc{\bbU}{\bb{U}}\nc{\bbV}{\bb{V}}\nc{\bbW}{\bb{W}}\nc{\bbX}{\bb{X}}
\nc{\bbY}{\bb{Y}}\nc{\bbZ}{\bb{Z}}
\nc{\mbf}[1]{{\mathbf #1}}
\nc{\bfA}{\mbf{A}}\nc{\bfB}{\mbf{B}}\nc{\bfC}{\mbf{C}}\nc{\bfD}{\mbf{D}}
\nc{\bfE}{\mbf{E}}\nc{\bfF}{\mbf{F}}\nc{\bfG}{\mbf{G}}\nc{\bfH}{\mbf{H}}
\nc{\bfI}{\mbf{I}}\nc{\bfJ}{\mbf{J}}\nc{\bfK}{\mbf{K}}\nc{\bfL}{\mbf{L}}
\nc{\bfM}{\mbf{M}}\nc{\bfN}{\mbf{N}}\nc{\bfO}{\mbf{O}}\nc{\bfP}{\mbf{P}}
\nc{\bfQ}{\mbf{Q}}\nc{\bfR}{\mbf{R}}\nc{\bfS}{\mbf{S}}\nc{\bfT}{\mbf{T}}
\nc{\bfU}{\mbf{U}}\nc{\bfV}{\mbf{V}}\nc{\bfW}{\mbf{W}}\nc{\bfX}{\mbf{X}}
\nc{\bfY}{\mbf{Y}}\nc{\bfZ}{\mbf{Z}}
\nc{\bfa}{\mbf{a}}\nc{\bfb}{\mbf{b}}\nc{\bfc}{\mbf{c}}\nc{\bfd}{\mbf{d}}
\nc{\bfe}{\mbf{e}}\nc{\bff}{\mbf{f}}\nc{\bfg}{\mbf{g}}\nc{\bfh}{\mbf{h}}
\nc{\bfi}{\mbf{i}}\nc{\bfj}{\mbf{j}}\nc{\bfk}{\mbf{k}}\nc{\bfl}{\mbf{l}}
\nc{\bfm}{\mbf{m}}\nc{\bfn}{\mbf{n}}\nc{\bfo}{\mbf{o}}\nc{\bfp}{\mbf{p}}
\nc{\bfq}{\mbf{q}}\nc{\bfr}{\mbf{r}}\nc{\bfs}{\mbf{s}}\nc{\bft}{\mbf{t}}
\nc{\bfu}{\mbf{u}}\nc{\bfv}{\mbf{v}}\nc{\bfw}{\mbf{w}}\nc{\bfx}{\mbf{x}}
\nc{\bfy}{\mbf{y}}\nc{\bfz}{\mbf{z}}

\newcommand{\G}{\mathbb{G}}
\newcommand{\op}{\text{op}}

\nc{\mcal}[1]{{\mathcal #1}}
\nc{\calA}{\mcal{A}}\nc{\calB}{\mcal{B}}\nc{\calC}{\mcal{C}}\nc{\calD}{\mcal{D}}
\nc{\calE}{\mcal{E}} \nc{\calF}{\mcal{F}}\nc{\calG}{\mcal{G}}\nc{\calH}{\mcal{H}}
\nc{\calI}{\mcal{I}}\nc{\calJ}{\mcal{J}}\nc{\calK}{\mcal{K}}\nc{\calL}{\mcal{L}}
\nc{\calM}{\mcal{M}}\nc{\calN}{\mcal{N}}\nc{\calO}{\mcal{O}}\nc{\calP}{\mcal{P}}
\nc{\calQ}{\mcal{Q}}\nc{\calR}{\mcal{R}}\nc{\calS}{\mcal{S}}\nc{\calT}{\mcal{T}}
\nc{\calU}{\mcal{U}}\nc{\calV}{\mcal{V}}\nc{\calW}{\mcal{W}}\nc{\calX}{\mcal{X}}
\nc{\calY}{\mcal{Y}}\nc{\calZ}{\mcal{Z}}
\nc{\fA}{\frak{A}}\nc{\fB}{\frak{B}}\nc{\fC}{\frak{C}} \nc{\fD}{\frak{D}}
\nc{\fE}{\frak{E}}\nc{\fF}{\frak{F}}\nc{\fG}{\frak{G}}\nc{\fH}{\frak{H}}
\nc{\fI}{\frak{I}}\nc{\fJ}{\frak{J}}\nc{\fK}{\frak{K}}\nc{\fL}{\frak{L}}
\nc{\fM}{\frak{M}}\nc{\fN}{\frak{N}}\nc{\fO}{\frak{O}}\nc{\fP}{\frak{P}}
\nc{\fQ}{\frak{Q}}\nc{\fR}{\frak{R}}\nc{\fS}{\frak{S}}\nc{\fT}{\frak{T}}
\nc{\fU}{\frak{U}}\nc{\fV}{\frak{V}}\nc{\fW}{\frak{W}}\nc{\fX}{\frak{X}}
\nc{\fY}{\frak{Y}}\nc{\fZ}{\frak{Z}}
\nc{\fa}{\frak{a}}\nc{\fb}{\frak{b}}\nc{\fc}{\frak{c}} \nc{\fd}{\frak{d}}
\nc{\fe}{\frak{e}}\nc{\fFf}{\frak{f}}\nc{\fg}{\frak{g}}\nc{\fh}{\frak{h}}
\nc{\fri}{\frak{i}}\nc{\fj}{\frak{j}}\nc{\fk}{\frak{k}}\nc{\fl}{\frak{l}}
\nc{\fm}{\frak{m}}\nc{\fn}{\frak{n}}\nc{\fo}{\frak{o}}\nc{\fp}{\frak{p}}
\nc{\fq}{\frak{q}}\nc{\fr}{\frak{r}}\nc{\fs}{\frak{s}}\nc{\ft}{\frak{t}}
\nc{\fu}{\frak{u}}\nc{\fv}{\frak{v}}\nc{\fw}{\frak{w}}\nc{\fx}{\frak{x}}
\nc{\fy}{\frak{y}}\nc{\fz}{\frak{z}}

\newtheorem{theorem}{Theorem}[section]
\newtheorem{lemma}[theorem]{Lemma}
\newtheorem{corollary}[theorem]{Corollary}
\newtheorem{prop}[theorem]{Proposition}

\newtheorem{assumption}[theorem]{Assumption}

\theoremstyle{definition}
\newtheorem{definition}[theorem]{Definition}
\newtheorem{example}[theorem]{Example}
\newtheorem{remark}[theorem]{Remark}

\newtheorem{thm}{Theorem}

\DeclareMathOperator{\im}{im} 
 \DeclareMathOperator{\id}{id}
 \DeclareMathOperator{\Sym}{Sym}
\DeclareMathOperator{\ch}{ch}

 \DeclareMathOperator{\GL}{GL}
\DeclareMathOperator{\Hom}{{Hom}}

\DeclareMathOperator{\Spec}{{Spec}}

\DeclareMathOperator{\Gm}{\bbG_m}

\newcommand{\sph}{\fs}

   \DeclareMathOperator{\sign}{sign}

\DeclareMathOperator{\Spf}{Spf}

\DeclareMathOperator{\Td}{Td}

\DeclareMathOperator{\inc}{in}
\DeclareMathOperator{\out}{out}

\DeclareMathOperator{\Rep}{Rep}
\DeclareMathOperator{\Res}{Res}

\DeclareMathOperator{\Sh}{Sh}

\newcommand{\surj}{\twoheadrightarrow}

\newcommand{\pt}{\text{pt}}

\newcommand{\PP}{\bbP}

\newcommand{\C}{\bbC}

\newcommand{\N}{\bbN}

\DeclareMathOperator{\fac}{fac}
\DeclareMathOperator{\pr}{pr}
 \newcommand{\ext}{\fe}
  \newcommand{\coop}{{coop}}
\newcommand{\loc}{loc}

\DeclareMathOperator{\SH}{\calS\calH}

\newenvironment{romenum}
{

\begin{enumerate}}{\end{enumerate}}

\newcount\cols
{\catcode`,=\active\catcode`|=\active
 \gdef\Young(#1){\hbox{$\vcenter
 {\mathcode`,="8000\mathcode`|="8000
  \def,{\global\advance\cols by 1 &}%
  \def|{\cr
        \multispan{\the\cols}\hrulefill\cr
        &\global\cols=2 }%
  \offinterlineskip\everycr{}\tabskip=0pt
  \dimen0=\ht\strutbox \advance\dimen0 by \dp\strutbox
  \halign
   {\vrule height \ht\strutbox depth \dp\strutbox##
    &&\hbox to \dimen0{\hss$##$\hss}\vrule\cr
    \noalign{\hrule}&\global\cols=2 #1\crcr
    \multispan{\the\cols}\hrulefill\cr%
   }
 }$}}
}

\begin{document}
\title{Cohomological Hall algebras and affine quantum groups}
\date{\today}

\author[Y.~Yang]{Yaping~Yang}
\address{School of Mathematics and Statistics, The University of Melbourne, 813 Swanston Street, Parkville VIC 3010, Australia}
\email{yaping.yang1@unimelb.edu.au}

\author[G.~Zhao]{Gufang~Zhao}
\address{Department of Mathematics,
University of Massachusetts, Amherst, MA, 01003, USA}
\email{zhao@math.umass.edu}


\begin{abstract}
We study the preprojective cohomological Hall algebra (CoHA) introduced by the authors in \cite{YZ1} for any quiver $Q$ and any one-parameter formal group $\mathbb{G}$. In this paper, we construct a comultiplication on the CoHA, making it a bialgebra.  We also construct the Drinfeld double of the CoHA. The Drinfeld double is a quantum affine algebra of the Lie algebra $\fg_Q$ associated to $Q$, whose quantization comes from the formal group $\bbG$. 
We prove, when the group $\bbG$ is the additive group, the Drinfeld double of the CoHA is isomorphic to the Yangian.
\end{abstract}

\maketitle
\section{Introduction}
In this paper, we construct comultiplications on certain cohomological Hall algebras, making them bialgebras. We also construct the Drinfeld double of these bialgebras. This gives a uniform way to construct both old and new affine-type quantum groups.

The cohomological Hall algebra involved, called the preprojective cohomological Hall algebra (CoHA for short) and denoted by $\calP(Q, A)$ or simply $\calP$,  is associated to a quiver $Q$ and an algebraic oriented cohomology theory $A$. The construction of the preprojective CoHA is given  in \cite{YZ1}, as a generalization of the $K$-theoretic Hall algebra of commuting varieties studied by Schiffmann-Vasserot in \cite{SV2}.   The preprojective CoHA is defined to be 
the $A$-homology of the moduli of representations of the preprojective algebra of $Q$. It has the same flavor as the cohomological Hall algebra associated to  quiver with potential defined by Kontsevich-Soibelman \cite{KS}. 
The authors also construct an action of $\calP(Q, A)$ on the $A$-homology of Nakajima quiver varieties associated to $Q$ in \cite{YZ1}.

Let $\fg_Q$ be the corresponding symmetric Kac-Moody Lie algebra of $Q$ and $\fb_Q\subset \fg_Q$ be the Borel subalgebra. 
As is shown in \cite{YZ1}, a certain spherical subalgebra in an extension of  $\calP(Q,A)$, denoted by $\calP^{\sph,\ext}(Q,A)$ or $\calP^{\sph,\ext}$, is a quantization of (a central extension of) $U(\fb_Q[\![u]\!])$, where the quantization depends on the underlying formal group law of $A$.
As in the case of quantized enveloping algebra of $\fg_Q$, the Drinfeld double of the Borel subalgebra should be the entire quantum group. This is the subject of the present paper. We construct a coproduct on $\calP(Q,A)$, and define the Drinfeld double, which is a quantization of  $U(\fg_Q[\![u]\!])$. 
Assume $A_{\Gm}(\pt)\cong \calO(\bbG)$ for some 1-dimensional algebraic group or formal group $\bbG$. When $\bbG$ is an affine algebraic group, the Drinfeld double of $\calP(Q,A)$ recovers the Drinfeld realization of the quantization of the Manin triple associated to $\bbP^1$ as described in \cite[\S~4]{D86}. 
When $\bbG$ is a formal group which does not come from an algebraic group, this gives new affine quantum groups which have not been studied in literature.  In the case when $\bbG$ is the elliptic curve, the method here gives a Drinfeld realization of the elliptic quantum group  of \cite{Fed} which is previously unknown.

In this paper, we focus on the purely algebraic description of $\calP(Q,A)$ in terms of shuffle algebra  without explicit reference to Nakajima quiver varieties. The shuffle algebra is reviewed in detail in 
 \S~\ref{sec:shuffle}.
 The algebra $\calP^{\sph,\ext}$ has two deformation parameters $t_1,t_2$, coming from a 2-dimensional torus action.
Let $\underline\calP^{\sph,\ext}$ be the quotient of $\calP^{\sph,\ext}$ by the torsion part over the $(\Gm)^2$-equivariant parameters $t_1,t_2$.
In \S~\ref{sec:coprod} of the present paper, 
we prove the following.
\begin{thm}\label{ThmIntr_A}
 There is a comultiplication $\Delta: \underline\calP^{\sph,\ext}\to \underline\calP^{\sph,\ext}\widehat{\otimes}\underline\calP^{\sph,\ext}$, making $\underline\calP^{\sph,\ext}$ a bialgebra. 
\end{thm}
Here the completion $\underline\calP^{\sph,\ext}\widehat{\otimes}\underline\calP^{\sph,\ext}$ is explained in \S~\ref{sec:coprod}. In particular, 
 the comultiplication $\Delta$ is a formal power series. However, it converges to a rational map from $\calP^{\sph,\ext}$ to $\underline\calP^{\sph,\ext}\otimes\underline\calP^{\sph,\ext}$, the explicit formula of which is given in  \S~\ref{sec:coprod}. 

In \S~\ref{sec:pairing},  we construct the Drinfeld double of the bialgebra $\underline\calP^{\sph,\ext}$, denoted by $D(\underline\calP^{\sph,\ext})$, which is a quantization of  $U(\fg_Q[\![u]\!])$ (in a sense explained in Remark~\ref{rmk:degeneration}), 
when $\bbG$ is an algebraic group over $\C$ endowed with a one form $\omega$. Since $\underline\calP^{\sph,\ext}$ is infinite-dimensional, we need to define a non-degenerate  bialgebra pairing in order to define the Drinfeld double. The natural  pairing is a residue pairing on a certain  ad\`ele version of $\underline\calP^{\sph,\ext}$, which is also a bialgebra with multiplication and comultiplication given by the same formulas as those on $\underline\calP^{\sph,\ext}$, the precise definition of which can be found in \S~\ref{subsec:paring_on_adele}.

In \S~\ref{sec;Yangian},  we study in detail this Drinfeld double $D(\underline\calP^{\sph,\ext})$ in the example
 when the group $\bbG$ is additive. 
A quotient of $D(\underline\calP^{\sph,\ext})$ is called the reduced Drinfeld double (defined in \S~\ref{symmetricYangian}),  denoted by $\overline{D}(\underline\calP^{\sph,\ext})$. 
\begin{thm}\label{introthm:Yangian}
Let $Y_\hbar(\fg_Q)$ be the Yangian endowed with the Drinfeld comultiplication.
Assume $Q$ has no edge-loops. Then there is a bialgebra epimorphism from $Y_\hbar(\fg_Q)$ to $\overline{D}(\underline\calP^{\sph,\ext})|_{t_1=t_2=\hbar/2}$.
When $Q$ is of finite type, this map is an isomorphism. 
\end{thm}
The morphism in Theorem \ref{introthm:Yangian} is expected to be an isomorphism for a more general class of quivers, which includes the affine Dynkin quivers. We will investigate this in a future publication \cite{GYZ}, based on the results of the present paper.
As far as we know, this is the first construction of the Yangian as a Drinfeld double of the  Borel subalgebra $Y_\hbar^{\geq 0}(\fg_Q)$. 
The double Yangian, which is the Drinfeld double of Yangian, can also be realized in a similar way. (See Remark~\ref{rmk:doubleYang} for details.)

Historically, there are a number of attempts to construct quantum groups using shuffle algebras. 
Rosso \cite{Ros} constructed the quantized enveloping algebra $U_q(\fg_Q)$ as a subalgebra in a suitable shuffle algebra, as an intrinsic description of $U_q(\fg_Q)$. This description turns out to be useful in the study of $q$-characters of quiver Hecke algebras \cite{Lec}.  
In a seminal note \cite{Gr2}, Grojnowski outlined the idea of constructing the positive part of the quantum loop algebra in terms of certain $K$-theoretic  Hall algebra, as well as the relation with Lusztig's construction of the quantized enveloping algebra using perverse sheaves \cite{L91}. As mentioned in \cite{Gr2}, Feigin in an unpublished work also has a construction of quantum groups in terms of symmetric polynomial ring. 
In \cite{Neg14} a $K$-theoretic shuffle algebra associated to the Jordan quiver has been studied. In particular, a comultiplication on this shuffle algebra has been constructed, making it a bialgebra.

In the present paper, we work in the generality that $\bbG$ is any 1-dimensional affine algebraic group or 1-parameter formal group.
Similar results hold when $\bbG$ is not necessarily affine. 
In \cite{YZ2}, we study in detail the case when $\bbG$ is the elliptic curve. 
We construct the sheafified shuffle algebra. The corresponding cohomology is the equivariant elliptic cohomology of \cite{Gr, GKV95}. When $\bbG$ is the universal elliptic curve over the (open) moduli space $\calM_{1,2}$ of 
genus 1 curves with 2 marked points, the space of suitable rational sections of this algebra object coincides with Felder's elliptic quantum group \cite{Fed, GTL15}. The precise statements and proofs will be published in \cite{YZ2}.

\subsection*{Acknowledgment}
We thank the anonymous referee for helpful comments. Most of the work was done when both authors were temporary faculty members at the University of Massachusetts, Amherst.

\section{The shuffle algebra}
\label{sec:shuffle}
In this section, we recall the algebraic description of the cohomological Hall algebra $\calP(Q, A)$ in terms of the shuffle algebra. The definition and  the multiplication formula of shuffle algebra is recalled in  \S~\ref{subsec:shuffle}. Relation with the cohomological Hall algebra $\calP(Q, A)$ is recalled in \S~\ref{subsec:relation with P}. 
\subsection{Formal group algebras}
We first fix the notations and general setup. 
Let $R$ be a commutative ring of characteristic zero. Let $\bbG$ be either a connected 1-dimensional affine algebraic group or a 1-dimensional formal group over $\Spec R$. The coordinate ring of $\bbG$ is denoted by $\mathbf{S}$, in particular, $\mathbf{S}$ is an $R$--algebra. Let $\fl\in \mathbf{S}$ be a local uniformizer of $\bbG$ at the identity section. 
Note that if $\bbG$ is an algebraic group, the expansion of the group multiplication of $\bbG$ with respect to the local uniformizer $\fl$ gives a formal group law.

For a natural number $n\in \bbN$, let $\fS_n$ be the symmetric group of $n$ letters.
When $\bbG$ is a 1-dimensional algebraic group, let $\bbG^n$ be the algebraic group whose coordinate ring is $\mathbf{S}^{\otimes n}$, denoted by $\mathbf{S}^n$; let $\bbG^{(n)}$ be the affine variety over $R$ whose coordinate ring is $(\mathbf{S}^{\otimes n})^{\fS_n}$, denoted by $\mathbf{S}^{(n)}$, i.e., $\bbG^{(n)}=\bbG^n/\fS^n$.
Similarly, when $\bbG$ is a 1-dimensional formal  group, we define $\bbG^n$ to be $\Spf( \mathbf{S}^{\widehat{\otimes}n})$, and define $\bbG^{(n)}$ to be $\Spf((\mathbf{S}^{\widehat{\otimes} n})^{\fS_n})$; here $ \mathbf{S}^{\widehat{\otimes}n}$ is still denoted by $\mathbf{S}^n$
and $(\mathbf{S}^{\widehat{\otimes} n})^{\fS_n}$ by $\mathbf{S}^{(n)}$.

More canonically,  let $\Lambda_n$ be a lattice of rank $n$, with a natural $\fS_n$-action. Then $\bbG^n$ is the set of homomorphisms of abelian groups $\Hom(\Lambda_n^\vee,\bbG)$, which naturally has the structure as an algebraic group. 
For any group homomorphism $\lambda:\Lambda_n \to \bbZ$, let $\chi_\lambda: \bbG^n=\Hom(\Lambda_n^\vee,\bbG)\to \bbG$ be the induced morphism. 
For the local uniformizer $\fl\in \mathbf{S}$, we denote $\fl(\lambda)\in \mathbf{S}^{n}$ the pullback of $\fl$ along $\chi_\lambda$.
Choose a natural coordinate $(x_1,\cdots,x_n)$ of $\Lambda_n$, such that the action of $\fS_n$ is permuting the coordinate $(x_1,\cdots,x_n)$. 
We have a group homomorphism $ \lambda_{ij}: \Lambda_n\to \bbZ$, defined as $(x_1,\cdots,x_n) \mapsto x_i-x_j$. 
Let $\fl(x_i-x_j)\in \mathbf{S}^n$ be the function on $\bbG^n$ associated to this group homomorphism. 
Similarly the function $\fl(x_i-x_j-t_1-t_2)\in \mathbf{S}^{n+2}$ is the function associated to the group homomorphism $\lambda_{ij, t_1, t_2}: \Lambda_n\oplus\bbZ^2\to \bbZ$ given by $(x_1,\cdots,x_n, t_1, t_2)\mapsto  x_i-x_j-t_1-t_2$. 

For simplicity, later on we will denote the set of integers $\{1,2,\dots,n\}$ by $[1,n]$, so that $(x_1,\cdots,x_n, t_1, t_2)$ is simply denoted by $((x_i)_{i\in [1,n]},t_1,t_2)$.

Note that $\fl(x_i+x_j)$ can be expressed in terms of the formal group laws. That is, let $+_F$ be the formal group law given by $(\bbG,\fl)$, then $+_F$ is determined by the property that $\fl(x_i+x_j)=\fl(x_i)+_F\fl(x_j)$.

\subsection{Shuffle algebra associated to a quiver}
\label{subsec:shuffle}
Let $Q$ be a quiver with vertex set $I$ and arrow set $H$. 
For each arrow $h\in H$, we denote by $\inc(h)$ (resp. $\out(h)$) the incoming (resp. outgoing) vertex of $h$.
We assume in this paper that $I$ and $H$ are finite sets. 
The opposite quiver $Q^{\op}$ is a quiver with the same set of vertices $I$, and the set of arrows $H^{\op}$ is endowed with a bijection $*:H\to H^{\op}$,  so that for each $h\in H$, the corresponding arrow $*(h)$, denoted by  $h^*$ for simplicity, has the reversed orientation as $h$. Thus, we have $\inc(h)=\out(h^*)$, and $\out(h)=\inc(h^*)$.
Let $\overline{A}=(a_{kl})_{k, l\in I}$ be the matrix whose $(k, l)$th entry is
$a_{kl}=\#\{h\in H^{\op} \mid \inc(h)=k, \out(h)=l\}$ for $k, l\in I$. Let $\overline C:=\id-\overline{A}$.
A dimension vector $v=(v^i)_{i\in I}\in \bbN^I$ of $Q$ is a collection of natural numbers, one for each vertex.

For each quiver $Q$, and each 1-dimensional affine algebraic group $\bbG$, we recall the definition of the shuffle algebra  associated to $Q$ and  $\bbG$ defined in \cite[\S 3]{YZ1}, denoted by $\SH_Q^\bbG$, or simply $\SH$ if both $Q$ and $\bbG$ are clear from the context. 
\Omit{The algebra $\SH_Q^\bbG$ has two quantization parameters $t_1, t_2$. Geometrically these two quantization parameters come from the two dimensional torus $T=\G_m^2$ action on the cotangent bundle of representation space of the quiver $Q$. }

Let $R_{t_1,t_2}$ be the coordinate ring of $\bbG^2=\Hom(\Lambda_2^{\vee}, \bbG)$, where $(t_1,t_2)$ are the coordinates of $\Lambda_2$. The shuffle algebra $\calS\calH$ is an $\bbN^I$-graded $R_{t_1,t_2}$-algebra. As an $R_{t_1,t_2}$-module, we have 
\[
\calS\calH=\bigoplus_{v\in\bbN^I}\calS\calH_v,\] where the degree $v$ piece $\calS\calH_v=\calO(\bbG^{(v)})$ is the coordinate ring of $\bbG^{(v)}:=\bbG^2\times\prod_{i\in I}\bbG^{(v^i)}$.
We will also consider $\SH_v$ as the space of $\fS_v=\prod_{i\in I}\fS_v$-invariant functions on $\bbG^{v}:=\bbG^2\times\prod_{i\in I}\bbG^{v^i}$, where $\fS_v$ acts on $\bbG^{v}$ in the natural way. 
The coordinates for $\bbG^{v}$ are denoted by $(t_1,t_2,(x^{(i)}_s)_{i\in I, s\in[1,v^i]})$, where $(t_1,t_2)$ are the coordinates for $\bbG^2$ and for each $i\in I$, the coordinates for $\bbG^{v^i}$ are $(x^{(i)}_1,\dots,x^{(i)}_{v^i})$.

In order to describe the multiplication of $\calS\calH$, we introduce some notations. For any dimension vector $v\in \bbN^I$, a partition of $v$ is a pair of collections $A=(A^i)_{i\in I}$ and $B=(B^i)_{i\in I}$, where  $A^{i}, B^{i} \subset [1, v^{i}]$ for any $i\in I$, satisfying the following conditions:
$A^i\cap B^i=\emptyset$, and  
 $A^{i} \cup B^{i}=[1, v^{i}]$. We use the notation $(A,B)\vdash v$ to mean $(A, B)$ is a partition of $v$. 
  
We also write $|A|=v$ if $|A^i|=v^i$ for each $i\in I$.  
For any two dimension vectors $v_1, v_2\in \bbN^I$ such that $v_1+v_2=v$, we introduce the notation  \[\bfP(v_1,v_2):=\{(A,B)\vdash v\mid |A|=v_1,|B|=v_2\}. \]
 There is a standard element $(A_o,B_o)$ in $\bfP(v_1,v_2)$ with 
$A_o^{i}:=[1, v_1^i]$, and 
$B_o^{i}:=[v_1^i+1, v_{1}^i+v_2^i]$ for any $i\in I$. This standard element will also be denoted by $([1,v_1],[v_1+1,v])$ for short.

Let $m:H\coprod H^{\op}\to \bbZ$ be a function, which for each $h\in H$ provides two integers $m_h$ and $m_{h^*}$. We will consider specializations of $(t_1,t_2)\in \bbG^2$ which are compatible with the function $m$ in the following sense. 
\begin{assumption} \label{Assu:WeghtsGeneral}
We consider specializations of $t_1$ and $t_2$ which are compatible with the integers $m_h, m_{h^*}$ for any $h\in H$, in the sense that $t_1^{m_h}t_2^{m_{h^*}}$ is a constant, i.e., does not depend on $h\in H$.
\end{assumption}

\begin{remark}\label{rmk:weights}Two examples of the integers $m_h,m_{h^*}$ satisfying Assumption~\ref{Assu:WeghtsGeneral} are the following. 
\begin{enumerate}
\item Let $t_1$ and $t_2$ are independent variables, but $m_h=m_{h^*}=1$ for any $h\in H$.
\item Specialize $t_1=t_2=\hbar/2$. For any pair of vertices $i$ and $j$ with arrows  $h_1, \dots, h_a$ from $i$ to $j$, the pairs of integers are $m_{h_p}=a+2-2p$ and $m_{h_p^*}=-a+2p$. 
\end{enumerate}
 \end{remark}
Consider the following two elements in the localization of 
$\calO(\mathbb{G}^{(v_1)}\times_{\bbG^2} \mathbb{G}^{(v_2)} )$.
Recall that the ring $\calO(\mathbb{G}^{(v_1)}\times_{\bbG^2} \mathbb{G}^{(v_2)} )$ is isomorphic to $\mathbf{S}^{(v_1)}\otimes_{R_{t_1,t_2}} \mathbf{S}^{(v_2)}$ if $\bbG$ is an algebraic group, and to $\mathbf{S}^{(v_1)}\widehat\otimes_{R_{t_1,t_2}} \mathbf{S}^{(v_2)}$ if $\bbG$ is a formal  group. 
For $(A, B)\in \bfP(v_1,v_2)$, we
define
\begin{align*}
& \fac_1(x_A|x_B):=\prod_{i \in I}
\prod_{s\in A^{i}}
\prod_{t\in B^i}
\frac{\fl(x^{(i)}_s-x^{(i)}_t+t_1+t_2)}{\fl(x^{(i)}_t-x^{(i)}_s)},\\
& \begin{small}
\fac_2(x_A|x_B):=\prod_{h\in H}\Big(
\prod_{s\in A^{\out(h)}}
\prod_{t\in B^{\inc(h)}}
\fl(x_t^{ \inc(h)}-x_s^{\out(h)}+ m_h t_1)
\prod_{s\in A^{\inc(h)}}
\prod_{t\in B^{\out(h)}}
\fl(x_t^{\out(h)}-x_s^{\inc(h)}+m_{h^*}t_2)
\Big).\end{small}
\end{align*}
For example, when $m_h=m_{h^*}=1$, the $\fac_2(x_A|x_B)$ simplifies as
\begin{equation*}
\fac_2(x_A|x_B)=\prod_{i, j\in I}
\prod_{s\in A^{i}}
\prod_{t\in B^{j}} \fl(x^{(j)}_t-x^{(i)}_s+ t_1)^{a_{ij}}
\fl(x^{(j)}_t-x^{(i)}_s+ t_2)^{a_{ji}}, 
\end{equation*} where $a_{ij}$ is the number of arrows from $i$ to $j$ of quiver $Q$.
We make the convention that if $A^i$ or $B^i$ is the empty set for some $i\in I$, then the corresponding factor
$\prod_{s\in A^{i}}
\prod_{t\in B^i}
\frac{\fl(x^{(i)}_s-x^{(i)}_t+t_1+t_2)}{\fl(x^{(i)}_t-x^{(i)}_s)}$ in $ \fac_1(x_A|x_B)$ is understood as $1$; and similarly for $\fac_2(x_A|x_B)$. Define $\fac(x_A|x_B)=\fac_1(x_A|x_B) \cdot \fac_2(x_A|x_B).$

We now describe a multiplication on $\SH$. As $\SH$ is an $\bbN^I$-graded algebra, we only need to describe the multiplication on homogeneous elements and then extend by linearity. 
For two dimension vectors $v_1, v_2\in \bbN^I$,
we describe the map $\SH_{v_1}\otimes\SH_{v_2}\to \SH_{v_1+v_2}$ as follows. 
We consider  $\SH_{v_1+v_2}$ as a subspace of $\calS\calH_{v_1}\otimes \calS\calH_{v_2}$, where the algebra embedding is given by pulling back functions via the natural projection $\bbG^{(v_1)}\times \bbG^{(v_2)} \to \bbG^{(v_1+v_2)}$. For $f_1\in \SH_{v_1}$, $f_2\in \SH_{v_2}$, we first  consider $f_1\otimes f_2\in \calS\calH_{v_1}\otimes \calS\calH_{v_2}$ as a function on $\bbG^{(v_1)}\times \bbG^{(v_2)}$, which in turn can be considered as a function on $\bbG^v$, invariant under the group $\fS_{v_1}\times \fS_{v_2}$. 
Let $(A_o, B_o)\in \bfP(v_1,v_2)$ be the standard element. Then 
$f_1\otimes f_2$ as a function on $\bbG^v$ can be written as $f_1(x_{A_o})\otimes f_2(x_{B_o})$. The product of $f_1(x_{A_o})$ and $f_2(x_{B_o})$  is then defined to be 
\footnote{There is a discrepancy between the notation here and \cite{YZ1}. The algebra denoted by $\SH$ in the present paper, i.e., the one whose multiplication has a sign $(-1)^{(v_2, \overline{C} v_1)} $ coming from the Euler-Ringel form, is  denoted by $\widetilde{\calS\calH}$ in \cite[\S 5.6]{YZ1}. We adapt this convention since the one without sign is not used in the present paper.}
\begin{align}\label{shuffle formula}
f_1(x_{A_o})\star f_2(x_{B_o})
=&\sum_{\sigma\in\Sh(v_1,v_2)}(-1)^{(v_2, \overline{C} v_1)}\sigma\Big(f_1(x_{A_o})\cdot f_2(x_{B_o})
\cdot \fac(x_{A_o}|x_{B_o})\Big)\notag\\
=&\sum_{(A, B)\in\bfP(v_1,v_2)} (-1)^{(v_2, \overline{C} v_1)} f_1(x_A)\cdot f_2(x_B)\cdot \fac(x_A|x_B) .
\end{align}
Here $\Sh(v_1,v_2)$ is the subset of $\prod_{i\in I}\fS_{v^i}$ consisting of elements preserving the order of elements in $A_o$ and $B_o$.
Note that although $\fac(x_A|x_B)$ has a simple pole, the product $f_1(x_{A_o})\star f_2(x_{B_o})$ is well-defined in $\mathbf{S}^{(v_1+v_2)}$.

\begin{example}
When $\bbG$ is a formal group whose coordinate ring is $\mathbf{S}=R[\![ x]\!]$. The shuffle algebra can be identified as 
\[
\calS\calH_{Q}^{\bbG}=\bigoplus_{v\in \N^I} R[\![t_1,t_2]\!]
[\![x^{(i)}_s]\!]_{i\in I, s=1,\dots, v^i}^{\fS_v}.\]
To describe the multiplication, it suffices to take $f_1(x') \in \calS\calH_{v_1}$, and $f_2(x'') \in \calS\calH_{v_2}$ and compute $f_1(x') \star f_2(x'') $ using \eqref{shuffle formula}. 
We consider $\calS\calH_{v_1}\otimes \calS\calH_{v_2}$ as a subspace of \[R[\![t_1,t_2]\!][\![x^{(i)}_s]\!]_{i\in I, s=1,\dots, (v_1+v_2)^i}\] by sending $x'^i_s$ to $x^i_s$, and $x''^i_t$ to $x^i_{t+v_1^i}$.
Formula \eqref{shuffle formula} then yields the following formula of $f_1(x')\star f_2(x'')$:
\begin{align*}
\sum_{\sigma\in\Sh(v_1,v_2)}&\sigma\Big(f_1f_2
\prod_{i\in I}\prod_{s=1}^{v_1^i}
\prod_{t=1}^{v_2^i}\frac{x\rq{}^i_s-_{F}x\rq{}\rq{}^i_t+_Ft_1+_Ft_2}{x\rq{}\rq{}^i_t-_{F}x\rq{}^i_s} 
\prod_{h\in H}\big(
\prod_{s=1}^{v_1^{\out(h)}}
\prod_{t=1}^{v_2^{\inc(h)}}
(x_t^{'' \inc(h)}-_{F}x_s^{'\out(h)}+_{F} m_h \cdot t_1)\\&
\prod_{s=1}^{v_1^{\inc(h)}}
\prod_{t=1}^{v_2^{\out(h)}}
(x_t^{''\out(h)}-_{F}x_s^{'\inc(h)}+_{F}m_{h^*}\cdot t_2)
\big)\Big)\in R[\![t_1,t_2]\!]
[\![x^i_s]\!]_{i\in I, s=1,\dots, (v_1+v_2)^i}^{\fS_{v_1+v_2}}.
\end{align*}

When $\bbG$ is the additive group $\bbG_a$, we have $\calS\calH_{Q}^{\bbG_a}=\bigoplus_{v\in \N^I} R[t_1,t_2]
[x^{(i)}_s]_{i\in I, s=1,\dots, v^i}^{\fS_v}.$ The multiplication is given by the same formula as above, while $+F$ is replaced by $+$, and $-_F$ is replaced by $-$.
\end{example}
For each $k\in I$, let $e_k$ be the dimension vector valued $1$ at vertex $k$ and zero otherwise. We define the \textit{spherical subalgebra} $\SH^{\bbG, \sph}_{Q}$ to be the subalgebra of $\SH^{\bbG}_{Q}$ generated by
$\SH_{e_k}$ as $k$ varies in $I$. 

\subsection{The extended shuffle algebra}
Let $\SH^0:=\Sym_{R}(\bigoplus_{i\in I} \mathbf{S}^{(e_i)})$ be the symmetric algebra on $\bigoplus_{i \in I}\mathbf{S}^{(e_i)}$. Here $\bigoplus_{i\in I}\mathbf{S}^{(e_i)}$ is a vector subspace in the coordinate ring of $\fh^{\bbG}:=\prod_{i\in I}\bbG$. In this subsection, we define the extended shuffle algebra using the $\SH^0$--action on $\SH$. 

For any $(A,B)\vdash v$, let $\bbG^A$ be $\bbG^2\times\prod_{i\in I}\bbG^{(|A^i|)}$ and similarly we have $\bbG^B$. There is a natural projection \[\bbG^A\times\bbG^B\to \bbG^{(v)}.\]
A function $P$ on $\bbG^{(v)}$ will be denoted by 
  $P (x_{A\cup B})$. Its pullback via the projection $\bbG^A\times_{\bbG^2}\bbG^B\to \bbG^{(v)}$ is denoted by $P(x_A\otimes x_B)$. 
For any $(A,B)\vdash v$, we introduce the following rational function on $\bbG^A\times_{\bbG^2}\bbG^B$
\begin{equation}
 \widehat{\Phi}(z_{B}|z_{A}): =
\frac{\fac(z_{B}|z_{A})}{\fac(z_{A}|z_{B})}.
\end{equation}
In the following special case, the function $ \widehat{\Phi}(z_{B}|z_{A})$ has a simple formula, which appears in \cite[\S 5.4]{YZ1}. For $k\in I$, we consider a special $(A,B)\vdash v$ with $B^{k}=\{1\}$, $A^k=[2,v^k]$, and $B^{i}=\emptyset$, $A^{i}=[1, v^{i}]$ for $i \neq k$. 
In this case 
$
\widehat{\Phi}(z_{B}|z_{A})=
\frac{\fac(z_{e_k}|z_{A})}{\fac(z_{A}|z_{e_k})}
$ will be denoted by $
\widehat{\Phi}_k(z, |A|)$. 
Recall that $a_{ik}$ is the number of arrows in $H$ from $i$ to $k$. 
Let $c_{ik}$ be the $(i, k)$--entry of the Cartan matrix of the quiver $Q$. 
That is $c_{ik}=-a_{ik}-a_{ki}$ if $k\neq i$, and $2$ if $k=i$. \Omit{
When $m_h=m_{h^*}=1$, we have 
 \begin{align*}
\widehat{\Phi_k}(z, v) &=\frac{
\prod_{t\in [1, v^k]}
\frac{\fl(z-x^k_{t}+t_1+t_2)}{\fl(x^k_t-z)}
\cdot
\prod_{\beta\in I\backslash\{k\}}
\prod_{t\in A^{\beta}} \fl(x^\beta_t-z+ t_1)^{a_{k\beta}}
\fl(x^\beta_t-z+ t_2)^{a_{\beta k}}}
{
\prod_{s\in [1, v^k]}
\frac{\fl(x^k_s-z+t_1+t_2)}{\fl(z-x^k_s)}
\cdot
\prod_{\alpha\in I\backslash\{k\}}
\prod_{s\in A^{\alpha}}
 \fl(z-x^\alpha_s+ t_1)^{a_{\alpha k}}
\fl(z-x^\alpha_s+ t_2)^{a_{k \alpha}}}
\end{align*}
When the integers $m_h$ and $m_{h^*}$ are as in Remark~\ref{rmk:weights}(2), we have: 
\begin{align*}
&=\prod_{\alpha\in I}\prod_{t=1}^{v^{\alpha}}
\frac{ \fl(z-x^\alpha_t+ (c_{k\alpha}) \frac{\hbar}{2})
}{  \fl(z-x^\alpha_t- (c_{k\alpha})\frac{\hbar}{2})
} \cdot \left(
\prod_{t=1}^{v^k}\frac{\fl(z-x^k_t-\hbar) }{ \fl(x^k_t-z+\hbar) }
\frac{\fl(z-x^k_t) }{  \fl(x^k_t-z) }\cdot
\prod_{\alpha\in I\backslash\{k\}}\prod_{t=1}^{v^{\alpha}}
\frac{  \fl(x^\alpha_t-z-(c_{k\alpha})\frac{\hbar}{2}) }
{\fl(z-x^\alpha_t+ (c_{k\alpha}) \frac{\hbar}{2})}\right)
\end{align*}}
If $\fl$ is an odd function, then $\widehat{\Phi}_k(z, v)$ can be simplified as follows. 
\begin{displaymath} 
\widehat{\Phi}_k(z, v)= 
\left\{
     \begin{array}{lr}
       \Big(\prod_{i\in I\backslash \{k\}} \prod_{j=1}^{v^{i}}\frac{\fl(z- x_{j}^{(i)}-t_1)^{a_{ik}} \fl(z- x_{j}^{(i)}-t_2)^{a_{ki}}}{\fl(z- x_{j}^{(i)}+t_2)^{a_{ik}} \fl(z- x_{j}^{(i)}+t_1)^{a_{ki}}}\Big)\cdot 
 \prod_{j=1}^{v^{k}}\frac{\fl(z- x_{j}^{(k)}+t_1+ t_2) }{\fl(z- x_{j}^{(k)}-t_1- t_2)}, \\   \phantom{12345678901234567890123456789}\text{when  $m_h=m_{h^*}=1$ as in Remark~\ref{rmk:weights}(1)}; \\
       (-1)^{v^k}
\prod_{i\in I}\prod_{t=1}^{v^{i}}
\frac{ \fl(x^{(i)}_t-z- (c_{ki}) \frac{\hbar}{2})
}{  \fl(z-x^{(i)}_t- (c_{ki})\frac{\hbar}{2})
},    \text{when $m_h$ and $m_{h^*}$ are as in Remark~\ref{rmk:weights}(2)}.
     \end{array}
   \right.
   \end{displaymath}

\Omit{
When $m_h=m_{h^*}=1$ as in Remark~\ref{rmk:weights}(1), we have  \[
\widehat{\Phi}_k(z, v)= 
\Big(\prod_{i\in I\backslash \{k\}} \prod_{j=1}^{v^{i}}\frac{\fl(z- x_{j}^{(i)}-t_1)^{a_{ik}} \fl(z- x_{j}^{(i)}-t_2)^{a_{ki}}}{\fl(z- x_{j}^{(i)}+t_2)^{a_{ik}} \fl(z- x_{j}^{(i)}+t_1)^{a_{ki}}}\Big)\cdot 
 \prod_{j=1}^{v^{k}}\frac{\fl(z- x_{j}^{(k)}+t_1+ t_2) }{\fl(z- x_{j}^{(k)}-t_1- t_2)}. 
 \] 
When the integers $m_h$ and $m_{h^*}$ are as in Remark~\ref{rmk:weights}(2), we have 
\begin{align*}
\widehat{\Phi_k}(z, v)=
(-1)^{v^k}
\prod_{i\in I}\prod_{t=1}^{v^{i}}
\frac{ \fl(x^i_t-z- (c_{ki}) \frac{\hbar}{2})
}{  \fl(z-x^i_t- (c_{ki})\frac{\hbar}{2})
}. 
\end{align*}}

We define a $\SH^0$ action on  $\SH$ as follows.
For $k\in I$, consider the rational function $H_k(w)=1+\frac{1}{\fl(h^{(k)})-\fl(w)}$ on $\bbG^{(e_k)}\times \bbG\subseteq \fh^\bbG\times\bbG$, where $h^{(k)}$ is the coordinate of $\bbG^{(e_k)}$ and $w$ is the coordinate of the second $\bbG$-factor.
Expanding $H_k(w)$ with respect of the local uniformizer  $\fl(w)$ of the  second $\bbG$-factor,  we can consider $H_k(w)$ as an element in the completion $ \mathbf{S}^{(e_k)}[\![ w ]\!]$, where $\mathbf{S}^{(e_k)}=\calO(\bbG^{(e_k)})$. 
For any $g \in \SH_{v}$, and any $k\in I$, we define
\begin{equation}\label{eqn:conj}
H_{k}(w) g  H_k(w)^{-1}:=g \widehat{\Phi_k}(w, v).
\end{equation}

\begin{lemma}\cite[Lemma 5.8]{YZ1}
\begin{enumerate}
\item The action of $\SH^0$  on $\SH$ is well-defined.
\item Furthermore, the action of $\SH^0$ on $\SH$ is by algebra homomorphisms.
\end{enumerate}
\end{lemma} 
\begin{definition}
Define the extended shuffle algebra associated to $\bbG$ and $Q$ to be the algebra $\SH^{\ext}=\SH^{0}\ltimes  \SH^{\bbG}_{Q}$. 
The spherical subalgebra in the extended shuffle algebra is 
$\SH^{\sph,\ext}=\SH^{0} \ltimes  \SH^{\bbG, \sph}_{Q}$.
\end{definition}
More concretely, $\SH^{0} \ltimes  \SH^{\bbG, \sph}_{Q}\cong \SH^{0} \otimes  \SH^{\bbG, \sph}_{Q}$ as vector spaces;
 the multiplication of $H_1\otimes g_1$ and $H_2\otimes g_2$ for $H_i\otimes g_i\in \SH^{0}\ltimes  \SH^{\bbG}_{Q}$ $i=1, 2$
 is given by $(H_1\otimes g_1)\star (H_2\otimes g_2)=H_1H_2\otimes((H_2\cdot g_1)\star g_2)$, where $H_2\cdot g_1$ is the action of $H_2$ on $g_1$ via conjugation \eqref{eqn:conj}. Extending by linearity, we get the multiplication on the entire $\SH^{0} \ltimes  \SH^{\bbG, \sph}_{Q}$.

\subsection{Geometric interpretation}
\label{subsec:relation with P}
The shuffle algebra $\calS\calH_{Q}^{\bbG}$ defined above has a geometric interpretation which we now recall. This geometric interpretation is not used in the later part of this paper and only serves as a motivation, hence can be skipped for readers only interested in the algebraic aspects of this paper. 

We consider algebraic equivariant oriented cohomology theory $A$ in the sense of \cite[\S~2]{CZZ3}, \cite[\S~5.1]{ZZ14}, and  \cite[\S~1.2]{YZ1}, examples of which include \cite{HM, Kr, To, Th}. That is, an assignment of a commutative ring $A_G(X)$ to any smooth quasi-projective variety $X$ with a reductive group $G$-action. In particular, when $G$ is the trivial group, this yields an  algebraic oriented cohomology theory in the sense of \cite[Chapter~2]{LM}.

In this section, 
assume $A$ is an equivariant oriented cohomology theory such that  $A_{\Gm}(\pt)\cong \mathbf{S}$, and that the formal group law of $\bbG$ coincides with the formal group law associated to the
cohomology theory $A$.

 Fix $V=\{V^i\}_{i\in I}$ an $I$-tuple of $k$-vector spaces with dimension vector $\dim(V^i)=v^i$. 
Let $\Rep(Q, v):=\bigoplus_{h\in H}\Hom(V^{\out(h)},V^{\inc(h)})$ be the representation space of $Q$ with dimension vector $v$ and let $T^*\Rep(Q, v)=\Rep(Q, v)\times \Rep(Q^{\op}, v)$ be the cotangent bundle of $\Rep(Q, v)$. The algebraic group $G_{v}:=\prod_{i\in I}\GL(v^i)$ acts on $\Rep(Q,v)$ and $T^*\Rep(Q,v)$ by conjugation.  There is a $T=\G_m^2$-action on $T^*\Rep(Q, v)$ determined by the function $m: H \coprod H^{\op}\to \bbZ$ as follows. For each ordered pair $(i, j)\in H^2$, let $a$ be the number of arrows in $Q$ from vertex $i$ to $j$. Fix an ordering on these arrows from $i$ to $j$, and label them as $h_1, \dots, h_a$. The corresponding reversed arrows in $Q^{\op}$ are labelled by $h_1^*, \cdots, h_a^*$. 
For each such $h_p \in H$, and any $B=(B_p) \in \Hom(V^i, V^j)$, $B^*=(B^*_p) \in \Hom(V^j, V^i)$, 
define the $T=\G_m^2$-action by:
 \[
 t_1\cdot B_p:=t_1^{m_{h_p}} B_p, \,\  t_2\cdot B_p^*:=t_2^{m_{h_p^*}} B_p^*.
 \] 
As abelian groups, we have 
\[
\calS\calH^{\bbG}_{Q}= \bigoplus_{v\in \N^I} A_{G_v\times (\Gm)^2} (T^*\Rep(Q, v)). 
\]
There is a geometrically defined multiplication $m^S_{v_1,v_2}$ on $A_{G_v\times T} (T^*\Rep(Q, v))$, constructed in \cite[\S 4]{YZ1},  that coincides with the  algebraically defined multiplication \eqref{shuffle formula}.
 \Omit{ 
Fix two dimension vectors  $v_1, v_2$ with $v_1+v_2=v$. Fix $V_1\subset V$ an $I$-tuple vector spaces with dimension vector $\dim(V_1)=v_1$. We write $G:=G_{v}$, and $P \subset G_v$, the parabolic subgroup preserving the subspace $V_1$.

Let $Y$ be $\Rep(Q,v_1)\times \Rep(Q,v_2)$ and 
$Z:=G\times_P\{ (x, x^*)\in T^*\Rep(Q, v_1+v_2)\mid (x, x^*)(V_1)\subset V_1\}$. 
In \cite{YZ1}, we have the following Lagrangian correspondence of $G\times T$-varieties:
\begin{equation}\label{corr}
\xymatrix{
G\times_PT^*Y\ar@{^{(}->}[r]^{\iota}&T^*(G\times_PY)&Z \ar[l]_(0.3){\phi} \ar[r]^(0.3){\psi} & T^*\Rep(Q, v_1+v_2), 
}
\end{equation}
where $T^*(G\times_PY)$ can be identified as  the set 
\[
G\times_P\{
(c, x, x^*)\mid c \in \fp_v , x \in \Rep( Q, v_1)\times \Rep( Q, v_2),
x^*\in \Rep( Q^{\op}, v_1)\times \Rep( Q^{\op}, v_2),  [x, x^*]=\pr(c)
\}\]
where $\pr(c)$ is the projection of $c$ in $\fg_{v_1}\oplus \fg_{v_2}$. 
The maps in \eqref{corr} are given by
\begin{align*}
&\phi: \big((g, x, x^*) \mod P\big) \mapsto \big(g, [x, x^*], \pr(x), \pr(x^*)\big) \mod P, \\
&\psi: \big((g, x, x^*, i, j) \mod P\big)\mapsto \big( gxg^{-1}, gx^*g^{-1}\big). 
\end{align*}
The multiplication map $m_{v_1, v_2}^{S}$ of $\bigoplus_{v\in \N^I} A_{G_v\times T} (T^*\Rep(Q, v))$ is defined as the composition $\psi_*\circ \phi^*\circ \iota_*$. }

\begin{remark}
If $R$ contains $\bbQ$, then any formal group $\bbG$ is isomorphic to the additive formal group $\bbG_a$. However, this does not imply that $(\SH^{\bbG}, m^{\bbG}) \cong (\SH^{\bbG_a}, m^{\bbG_a})$, where $m^{\bbG}$ is the multiplication of  $\SH^{\bbG}$. Instead, there is an $R$-module isomorphism $\ch:\SH^{\bbG}\to \SH^{\bbG_a}$, which has the following property. For any $v_1,v_2\in\bbN^I$, there exists an element $\Td_{v_1,v_2}\in \SH^{\bbG_a}_{v_1}\otimes\SH^{\bbG_a}_{v_2}$, such that for any $f\otimes g\in \SH^{\bbG_a}_{v_1}\otimes\SH^{\bbG_a}_{v_2}$, we have 
\[
m^{\bbG_a}(\ch(f\otimes g)\cdot \Td_{v_1,v_2})=\ch(m^{\bbG}(f\otimes g)).\] The element $\Td_{v_1,v_2}$ can be chosen as $\Td_{v_1,v_2}=\frac{\fac^{\bbG}_{v_1,v_2}}{\fac^{\bbG_a}_{v_1,v_2}}$.
\end{remark}

The following is the precise relation between the shuffle algebra and the CoHA, which can be found in \cite[Theorem~C]{YZ1}. 
\begin{prop}\label{prop:Shuf_CoHA}
There is an algebra homomorphism from the extended CoHA $\calP^{\ext}(Q,A)$ to the extended shuffle algebra $\SH^{\ext}$, which induces an isomorphism 
\[\underline\calP^{\sph,\ext}(Q,A)\cong \SH^{\sph,\ext}\]
where $\underline\calP^{\sph,\ext}(Q,A)$ is the quotient of $\calP^{\sph,\ext}(Q,A)$ by the torsion part over $R_{t_1,t_2}$.
\end{prop}

\begin{remark}\label{rmk:degeneration}
Assume $Q$ is a quiver without edge-loops. 
When $\bbG$ is the additive group, it is shown in  \cite[Theorem~D]{YZ1} that $\SH^{\sph,\ext}|_{t_1=t_2=0}$  is a central extension of $U(\fb[\![u]\!])$, where $\fb\subset \fg$ is the Borel subalgebra. This central extension is trivial when $Q$ is of finite type, and is described in detail in \cite[\S~8.3.4]{YZ1} when $Q$ is of affine type. 

For general $\bbG$, for simplicity we take the formal completion of $\bbG$ in the case when $\bbG$ is an algebraic group. Assume $\bbG$ comes from an oriented cohomology theory. Then there is a natural cohomological grading on the base ring $R$ so that the natural classifying map from the Lazard ring to $R$ is homogeneous of degree zero. Let $\overline R$ be the image of the Lazard ring in $R$, which is non-negatively graded. There is an augmentation map to the zeroth degree piece $\overline R\to R^0$.  For example, when $\bbG=\bbG_m$ is multiplicative, the cohomology theory is the $K$-theory, in which case $R=\bbZ[\beta^\pm]$ where $\beta$ is the Bott periodicity; The image of the Lazard ring is $\bbZ[\beta^{-1}]$, which has an augmentation map $\bbZ[\beta^{-1}]\to \bbZ$ sending $\beta^{-1}$ to 0. On the quotient, the image of the formal group law is additive. Therefore, on this special fiber of $\Spec \overline{R}$, the algebra $\SH^{\sph,\ext}|_{t_1=t_2=0}$ becomes a central extension of $U(\fb[\![u]\!])$ as an algebra.  In this sense, we say $\SH^{\sph,\ext}$ is a quantization of $U(\fb[\![u]\!])$.
\end{remark}

\section{The comultiplication}\label{sec:coprod}
In this section, we construct the map $\Delta$ on a suitable localization of the extended shuffle algebra 
\[
\Delta: \SH_{v}^{\ext}\to \sum_{v_1+v_2=v} (\SH_{v_1}^{\ext} \otimes \SH_{v_2}^{\ext})_{\loc}. 
\]
Here  $\loc$ means the localization away from   the union of null-divisors of $\fac(x_A,x_B)$ over all $(A,B)\vdash v$. The ideal of this divisor in $\SH_v$ is denoted by  $I_{\fac}$. 

\subsection{Definition of the comultiplication}
On the symmetric algebra $\SH^0$, the map $\Delta$ is determined by
\begin{align}\label{coprodH}
&
\Delta(H_k(w))=H_k(w)\otimes H_k(w), k\in I.
\end{align}
This is understood in the following way. Recall that $H_k(w)=1+\frac{1}{\fl(h^{(k)})-\fl(w)}$ is a rational function on $\bbG^{(e_k)}\times \bbG\subseteq \fh^\bbG\times\bbG$. The function $H_k(w)\otimes H_k(w)$ is the rational function on $\bbG^{(e_k)}\times\bbG^{(e_k)}\times\bbG\subseteq \fh^\bbG\times\fh^\bbG\times\bbG$ defined as the product $p_1^*( H_k(w)) \cdot p_2^*( H_k(w))$, 
where $p_i: \bbG^{(e_k)}\times\bbG^{(e_k)}\times\bbG\to \bbG^{(e_k)}\times\bbG$ is the natural projection contracting the $i$-th $\bbG^{(e_k)}$-factor, for $i=1, 2$ and $p_i^*$ is the pullback along the projection $p_i$.  Here $w$ is the coordinate of the last $\bbG$-factor.  Expanding $H_k(w)\otimes H_k(w)$  with respect to $\fl(w)$, we get a well-defined element in a completion of $\SH^0\otimes\SH^0[\![w]\!]$.

We now define  $\Delta$ on $\SH$ by the following formula for an homogeneous element $P(x)\in \SH_v$; extending this formula by linearity gives a definition of $\Delta$ on the entire $\SH$.
\begin{equation}\label{eq:coprod}
\Delta (P(x))
=\sum_{\{v_1+v_2=v\}}
(-1)^{(v_2, \overline{C} v_1)}
\frac{ H_{[1,v_1]}(x_{[v_1+1,v]}) P(x_{[1,v_1]}\otimes x_{[v_1+1,v]})}{\fac(x_{[v_1+1,v]}| x_{[1,v_1]})},
\end{equation}
which is a rational function on $(\fh^{\bbG}\times\bbG^{(v_1)})\times(\bbG^{(v_2)}).$
Here
$ H_A(x_B):=\prod_{k\in I} \prod_{ j\in B^{k}} 
H_k (x_j^{(k)})$. The function $H_k( x_{j}^{(k)})$ is defined by pulling back 
the function $H_k(w)$ of $\bbG^{(e_k)}\times \bbG$ along the identification $\bbG^{(e_k)}\times \bbG_j\cong \bbG^{(e_k)}\times \bbG$. Note that in the inclusion
$\bbG^{(e_k)}\times \bbG_j\subset (\fh^{\bbG}\times\bbG^{(v_1)})\times(\bbG^{(v_2)})$, the $\bbG^{(e_k)}$-factor lies in the  $\fh^{\bbG}$-factor, and the $\bbG_j$-factor lies in the $\bbG^{(v_2)}$-factor. 
Expanding each $H_k(x_j^{(k)})$ with respect to $\fl(x_j^{(k)})$, we get  $H_{[1,v_1]}(x_{[v_1+1,v]})$ as a well-defined element in the completion 
\begin{equation}\label{eqn:completion}
\sum_{v_1+v_2=v} \SH_{v_1}^{\ext}\widehat{\otimes} \SH_{v_2}^{\ext}.
\end{equation}
Here the completion of $\SH_{v_1}^{\ext}\otimes \SH_{v_2}^{\ext}$ is taken with respect to the  powers of $\frac{\fl(h^{(k)})}{\fl(x^{(s)})}$ where $k\in [1,v_1]$ and $s\in [v_1+1,v]$, as illustrated in the following.
\begin{example}\label{ex:completion}
When $\bbG=\bbG_a$, $I=\{\pt\}$,  $v_1=v_2=1$, we have $H_1(x_2)=1+\frac{1}{h-x_2}$ as a rational function on $\bbA^2$, where the coordinates of $\bbA^2$ are $(h,x_2)$. Expanding with respect to powers of $x_2$ we get $H_1(x_2)=1+\sum_{i=0}^\infty x_2^ih^{-i-1}$. Note that this is the power series expansion around the neighborhood of $(\infty,0)\in\bbP^1\times\bbP^1$, which does not lie in $\bbA^2\subseteq \bbP^1\times\bbP^1$.
\end{example}

Note that $\SH^{\ext}\otimes \SH^{\ext}$ is an algebra, with the multiplication structure 
$(a_1\otimes b_1)\star (a_2\otimes b_2):=(a_1\star a_2)\otimes (b_1\star b_2)$, for $a_i, b_i\in \SH_{v_1}^{\ext}$, $i=1, 2$.
The comultiplication of a general element $\sum_i (H_i, P_i)\in \SH^{\ext}$ is defined to be
\begin{equation}
\label{coprod and algebra}
\Delta(\sum_i (H_i, P_i))
=\sum_i \Delta(H_i)\star \Delta(P_i), 
\end{equation}
where $\Delta(H_i), \Delta(P_i)$ are determined by \eqref{coprodH} and  \eqref{eq:coprod}.

In particular,
 when $v=e_k$ for some $k\in I$, either $A$ or $B$ is empty, and  $P(x_{A\cup B})=f(x_k)$ a function with only one variable $x^{(k)}$. 
We have
\begin{equation}\label{eq:delta f}
\Delta(f(x^{(k)}))=H_k(x^{(k)}) \otimes  f(x^{(k)})+f(x^{(k)})\otimes 1.
\end{equation}

\begin{theorem}\label{thm:coprod}The map $\Delta$ from $\SH^{\ext}$ to a localization of $\SH^{\ext}\otimes \SH^{\ext}$ defined above satisfies the following:
\begin{romenum}
\item\label{item1}
$\Delta(P\star Q)=\Delta(P)\star\Delta(Q)$ for any $P,Q\in \SH^{\ext}$;
\item \label{item2}
$(\mathrm{id}_{\SH^{\ext}} \otimes \Delta) \circ \Delta = (\Delta \otimes \mathrm{id}_{\SH^{\ext}}) \circ \Delta$;
\item  \label{item3}
Let $\epsilon: \SH^{\ext}\to R$ be the augmentation map, i.e.,  $\epsilon:\SH^0\to R$ is the augmentation map of the symmetric algebra, $\epsilon=\id: \SH_{0} =R\to R$,  and $\SH_{v}\mapsto 0$ for $v\neq 0$. 
Then,  $
(\mathrm{id}_{\SH^{\ext}} \otimes \epsilon) \circ \Delta = \mathrm{id}_{\SH^{\ext}} = (\epsilon \otimes \mathrm{id}_{\SH^{\ext}}) \circ \Delta.$
\end{romenum}
\end{theorem}

Note that by formula~\eqref{eq:delta f} of $\Delta(f)$  for $f\in \SH_{e_k}$,  $\Delta(\SH_{e_k}) $ lies in $\SH^{\ext}\otimes \SH^{\ext}$ localized at poles of $H_{[1,v_1]}(x_{[v_1+1,v]})$.  Expanding with respect to powers of $\fl(x_{[v_1+1,v]})$, we get a well-defined element in $\SH^{\ext}\widehat{\otimes} \SH^{\ext}$, where the completion is in the same sense as \eqref{eqn:completion}. Consequently, by Theorem~\ref{thm:coprod}, we have \[
\Delta(\SH^{\sph}_v)\subset 
\sum_{v_1+v_2=v} \SH_{v_1}^{\ext}\widehat{\otimes} \SH_{v_2}^{\ext}.\]
Therefore, we get the following direct corollary to Theorem~\ref{thm:coprod}.
\begin{corollary}
The map $\Delta: \SH^{\sph,\ext}\to \SH^{\sph,\ext}\widehat{\otimes}\SH^{\sph,\ext}$ is a well-defined algebra homomorphism which is coassociative, making $\SH^{\sph,\ext}$ a bialgebra with counit $\SH^{\sph,\ext}\to R$ being the restriction of $\epsilon:\SH^{\ext}\to R$. 
\end{corollary}
Taking Proposition~\ref{prop:Shuf_CoHA} into account, the above corollary implies Theorem~\ref{ThmIntr_A}.

\subsection{Proof of Theorem~\ref{thm:coprod}}
\begin{proof}[Proof of Theorem~\ref{thm:coprod}\eqref{item1}]
We start by proving  \eqref{item1}. 
By \eqref{coprodH}, it is clear that both sides of \eqref{item1} coincide when restricted on $\SH^0$. 

We now prove  \eqref{item1} when restricted to $\SH$.
Without loss of generality, we take $P\in \SH_{v_1}, Q\in \SH_{v_2}$ to be homogenous elements. 
Note that we have the same sign $(-1)^{(v_2, \overline{C} v_1)}$ in both the multiplication  \eqref{shuffle formula} and the comutiplication  \eqref{eq:coprod} of $\SH^{\ext}$. Therefore, to show  \eqref{item1}, we could drop the sign $(-1)^{(v_2, \overline{C} v_1)}$ in both $\star$ and $\Delta$.

We need the following notations.
For a pair of dimension vector $(v_1', v_2')$, with $v_1'+v_2'=v$, for $(A, B)\in\bfP(v_1, v_2)$, 
we write
\[
A_1:=A\cap [1, v_1'], A_2:=A\cap [v_1'+1, v_1'+v_2'], 
B_1:=B\cap [1, v_1'], B_2:=B\cap [v_1'+1, v_1'+v_2']. 
\]
Thus we have
\[
A=A_{1}\sqcup A_{2}, B=B_{1}\sqcup B_{2}, \text{and}\,\ 
A_1\sqcup B_1=[1, v_1'], A_2\sqcup B_2=[v_1'+1, v_1'+v_2']. 
\]
We compute $\Delta(P\star Q)$ using  \eqref{shuffle formula} and \eqref{eq:coprod}. 
\begin{align}
 &\Delta(P\star Q) \notag\\
 =&
\sum_{(A, B)\in\bfP(v_1,v_2)} 
\Delta(P(z_A)\cdot Q(z_B)\cdot \fac(z_A, z_B)) \notag\\
=&\sum_{(A, B)\in\bfP(v_1,v_2)} 
\sum_{ v_1'+v_2'=v}
\frac{ H_{A_1 \sqcup B_1}(z_{A_2 \sqcup B_2})
P(z_{A_1} \otimes z_{A_2})
Q(z_{B_1} \otimes z_{B_2})
\fac(z_{A_1} \otimes z_{A_2}|z_{B_1} \otimes z_{B_2})}
{ \fac( z_{[v_1'+1, v_1'+v_2']}|z_{[1, v_1']})}
\notag\\
=&\sum_{(A, B)\in\bfP(v_1,v_2)} 
\sum_{ v_1'+v_2'=v}
\frac{ H_{A_1}(z_{A_2})H_{B_1}(z_{B_2})
P(z_{A_1} \otimes z_{A_2})
Q(z_{B_1} \otimes z_{B_2})
\fac(z_{A_1\sqcup A_2 }|z_{B_1\sqcup B_2})}
{ \fac( z_{A_2\sqcup B_2}|z_{A_1\sqcup B_1})}
\notag\\
=&\sum_{ v_1'+v_2'=v}
\sum_{(A, B)\in\bfP(v_1,v_2)}
\frac{ H_{A_1}(z_{A_2})
P(z_{A_1} \otimes z_{A_2})H_{B_1}(z_{B_2})
Q(z_{B_1} \otimes z_{B_2})
}
{ \fac(z_{A_2\sqcup B_2}|z_{A_1\sqcup B_1})}
\widehat{\Phi}(z_{B_2}| z_{A_1}) \fac(z_{A_1\sqcup A_2}|z_{B_1\sqcup B_2}).
\label{eq:in coproduct}
\end{align}
The last equality is obtained from the identity $
H_{B_1}(z_{B_2}) P(z_{A_1})
=P(z_{A_1})H_{B_1}(z_{B_2}) \widehat{\Phi}(z_{B_2}| z_{A_1}).$
Recall that by definition $\widehat{\Phi}(z_{B_2}| z_{A_1})=
\frac{\fac(z_{B_2}|z_{A_1})}{\fac(z_{A_1}|z_{B_2})},$ and hence
\[
\frac{\widehat{\Phi}(z_{B_2}| z_{A_1}) \fac(z_{A_1\sqcup A_2}|z_{B_1\sqcup B_2})}{ \fac(z_{A_2\sqcup B_2}|z_{A_1\sqcup B_1})}
= \frac{
\fac(z_{A_1}|z_{B_1})
\fac(z_{A_2}|z_{B_2})
}{ 
\fac(z_{A_2}|z_{A_1})
\fac(z_{B_2}|z_{B_1}).
}\]
Plugging this  into  \eqref{eq:in coproduct},  we get 
\begin{align}
\label{eq:coproduct2}
\sum_{ v_1'+v_2'=v}
\sum_{(A, B)\in\bfP(v_1, v_2)} 
&\frac{H_{A_1}(z_{A_2})
P(z_{A_1} \otimes z_{A_2})  }{\fac(z_{A_2}|z_{A_1})}
\frac{H_{B_1}(z_{B_2})
Q(z_{B_1} \otimes z_{B_2}) }
{\fac(z_{B_2}|z_{B_1})}
\fac(z_{A_1}|z_{B_1})
\fac(z_{A_2}|z_{B_2}).
\end{align}

We then compute $\Delta(P)\star\Delta(Q)$.
We have
\begin{align}
&\Delta(P(z_{A_o}))\star\Delta(Q(z_{B_o})) \notag\\
=&\Big(\sum_{A_{o}=A_{1o}\sqcup A_{2o}}
\frac{ H_{A_{1o}}(x_{A_{2o}}) P(x_{A_{1o}}\otimes x_{A_{2o}})}{\fac(x_{A_{2o}}| x_{A_{1o}})}\Big)
\star
\Big(\sum_{B_{o}=B_{1o}\sqcup B_{2o}}
\frac{ H_{B_{1o}}(x_{B_{2o}}) Q(x_{B_{1o}}\otimes x_{B_{2o}})}{\fac(x_{B_{2o}}| x_{B_{1o}})}\Big)
\notag\\
=&\sum_{\substack{
A_{1o}\sqcup A_{2o}=A_{o}, \\B_{1o}\sqcup B_{2o}=B_{o}}}
\sum_{\substack{(A_1, B_1)\in \bfP(|A_{1o}|, |B_{1o}| ) \\
(A_2, B_2)\in \bfP(|A_{2o}|,|B_{2o}| )
}}
\frac{H_{A_1}(z_{A_2})
P(z_{A_1} \otimes z_{A_2})  }{\fac(z_{A_2}|z_{A_1})}
\frac{H_{B_1}(z_{B_2})
Q(z_{B_1} \otimes z_{B_2}) }
{\fac(z_{B_2}|z_{B_1})}
\fac(z_{A_1}|z_{B_1})
\fac(z_{A_2}|z_{B_2}), 
\label{eq:coprod3}
\end{align}
where the first equality follows from the coproduct formula \eqref{eq:coprod}, and the second equality follows from \eqref{shuffle formula}. 
Here in the indices of summations we write $A_{1o}\sqcup A_{2o}=A_{o}$ to mean a standard partition of $A_o$, i.e.,
$A_{1o}=[1, w_1]$ and $A_{2o}=[w_1+1, v_1]$ for some $1 \leq w_1 \leq v_1$; Similarly, $B_{1o}\sqcup B_{2o}=B_{o}$ stands for a standard partition of $B_{o}$.
Switching the order of $A_{2o}$ and $B_{1o}$, we  write $A_{1o} \sqcup B_{1o}$ as $[1, v_1']$, and $A_{2o} \sqcup B_{2o}$ as $[v_1'+1, v]$. 
The set \[\{(A_o=A_{1o}\sqcup A_{2o}, B_o=B_{1o}\sqcup B_{2o},(A_1,B_1),(A_2,B_2)) \mid (A_1,B_1)\in\bfP(|A_{1o}|, |B_{1o}| ) , (A_2,B_2)\in  \bfP(|A_{2o}|,|B_{2o}| )
\}\] is in bijection with the set \[\{(v_1', v_2' ,(A,B))\mid v_1'+v_2'=v, (A,B)\in \bfP(v_1,v_2)\},\] where the bijection is given by 
taking $A=A_1\sqcup A_2$ and $B=B_1\sqcup B_2$. Comparing formula \eqref{eq:coprod3} with the formula \eqref{eq:coproduct2}, we obtain 
$\Delta(P\star Q)=\Delta(P)\star\Delta(Q)$. 

For general elements $P,Q\in \SH^{\ext}$,  \eqref{item1} follows from Lemma~\ref{lem:reduce} below. 
\end{proof}
\begin{lemma}\label{lem:reduce}
For any $P\in \SH$, and $H\in \SH^0$, we have 
$\Delta(P\star H)=\Delta(P)\star \Delta(H)$. 
\end{lemma}
\begin{proof}\Omit{
Suppose the claim is true for $H_1, H_2\in \SH^0$. We now show it is also true for $H_1 H_2$. 
That is, suppose for any $P\in \SH$, one has $\Delta(P\star H_i)=\Delta(P)\star \Delta(H_i)$, for $i=1, 2$. We check that $\Delta(P\star (H_1H_2))=\Delta(P)\star \Delta(H_1H_2)$. Indeed, on one hand, we have
\begin{align*}
&\Delta(P\star (H_1 H_2))
=\Delta((H_1H_2)\star (H_1H_2)(P))
=\Delta(H_1H_2) \star \Delta((H_1H_2)(P)). 
\end{align*}
On the other hand, by the formula \eqref{coprodH}, it is clear that the two sides of \eqref{item1} coincide when restricted on $\SH^0$. Using $\Delta(H_1H_2)=\Delta(H_1)\Delta(H_2)$, and the claim for $H_1, H_2$, we have
\begin{align*}
\Delta(P)\star \Delta(H_1H_2)
=&\Delta(P)\star \Delta(H_1)\star \Delta(H_2)
=\Delta(P\star H_1)\star \Delta(H_2)\\
=&\Delta(H_1 \star H_1(P))\star \Delta(H_2)
=\Delta(H_1) \star \Delta(H_1(P))\star \Delta(H_2)\\
=&\Delta(H_1) \star \Delta(H_1(P)\star H_2)
=\Delta(H_1) \star \Delta(H_2 \star H_2H_1(P))\\
=&\Delta(H_1H_2) \star \Delta((H_1H_2)(P)). 
\end{align*}
Therefore, $\Delta(P\star (H_1 H_2))=\Delta(P)\star \Delta(H_1H_2)$. As a consequence, to show that lemma, }
Without loss of generality, we could choose $H=H_{k}(w)$. The comultiplication $\Delta$ is $R$-linear. It suffices to choose $P$ to be the homogeneous element in $\SH_v$. Now, for any $P\in \SH_{v}$, and $H=H_k(w)$. We compute $\Delta(P)\star \Delta(H)$ as follows. As before, we drop the sign $(-1)^{(v_2, \overline{C} v_1)}$ in the formula of $\Delta(P)$ for the convenience of the notation.  

By \eqref{coprodH} and \eqref{eq:coprod}, we have
\begin{align*}
\Delta(P)\star \Delta(H)
=&\sum_{(A_o, B_o)\vdash v}
\frac{ H_{A_o}(z_{B_o}) P(z_{A_o}\otimes z_{B_o})}{\fac(z_{B_o}| z_{A_o})}\star 
\big(H_k(w)\otimes H_k(w)\big)\\
=&\sum_{(A_o, B_o)\vdash v}\big(H_k(w)\otimes H_k(w)\big)\star
\Big(
\frac{ H_{A_o}(z_{B_o}) P(z_{A_o}\otimes z_{B_o})}{\fac(z_{B_o}| z_{A_o})}
\widehat{\Phi}_k(w| z_{A_0})\widehat{\Phi}_k(w| z_{B_0})\Big)\\
=&\sum_{(A_o, B_o)\vdash v}\big(H_k(w)\otimes H_k(w)\big)\star
\Big(
\frac{ H_{A_o}(z_{B_o}) P(z_{A_o}\otimes z_{B_o})}{\fac(z_{B_o}| z_{A_o})}
\widehat{\Phi}_k(w| z_{A_0\cup B_o})\Big)\\
=&\big(H_k(w)\otimes H_k(w)\big)\star
\sum_{(A_o, B_o)\vdash v}
\Big(
\frac{ H_{A_o}(z_{B_o}) P(z_{A_o}\otimes z_{B_o}) \widehat{\Phi}_k(w| z_{A_o}\otimes z_{B_o})}{\fac(z_{B_o}| z_{A_o})}\Big)
\\
=&
\Delta(H_k(w))\star \Delta(H_k(w)(P))
=\Delta(H \star H(P))
=\Delta(P\star H).
\end{align*}
This completes the proof. 
\end{proof}
To prove Theorem~\ref{thm:coprod}  \eqref{item1} in general, by linearity, without loss of generality we assume 
$P=H_1\star P_1\in \SH^{\ext}$ and $Q=H_2\star P_2\in \SH^{\ext}$ for $H_i\in \SH^0$ and $P_i\in \SH$. We have, on one hand, 
\begin{align*}
&\Delta((H_1\star P_1)\star (H_2\star P_2))
=\Delta(H_1H_2\star H_2(P_1)\star P_2)
=\Delta(H_1)\star \Delta(H_2)\star \Delta(H_2(P_1))\star \Delta(P_2).
\end{align*}
On the other hand, using Lemma~\ref{lem:reduce}, we have
\begin{align*}
&\Delta(H_1\star P_1)\star \Delta(H_2\star P_2)
=\Delta(H_1)\star \Delta(P_1)\star \Delta(H_2)\star \Delta(P_2)
=\Delta(H_1)\star \Delta(P_1 \star H_2)\star \Delta(P_2)\\
=&\Delta(H_1)\star \Delta(H_2\star H_2(P_1))\star \Delta(P_2)
=\Delta(H_1)\star \Delta(H_2)\star \Delta(H_2(P_1))\star \Delta(P_2).
\end{align*}
This implies the equality $\Delta((H_1\star P_1)\star (H_2\star P_2))=\Delta(H_1\star P_1)\star \Delta(H_2\star P_2)$ in general.

\begin{proof}[Proof of Theorem~\ref{thm:coprod} \eqref{item2}]
Thanks to  \eqref{item1}, to prove  \eqref{item2} it suffices to show that both sides are equal when applied to  the generators of $\SH^{\ext}$.
By \eqref{coprodH}, it is clear that both sides of \eqref{item2} coincide when restricted on $\SH^0$. 

We now show \eqref{item2} when both sides are restricted on $\SH$.
Without loss of generality, we may assume $P\in \SH$ is a homogeneous element.  
For any $A, B, C$, the projection of $(\id \otimes \Delta) (\Delta(P))$ into the component $\SH_A\otimes \SH_B \otimes  \SH_C$ is denoted by $(\id \otimes \Delta) (\Delta(P))_{A,B,C}$. Similarly, we have $(\Delta \otimes \id) (\Delta(P))_{A,B,C}$.  With this notation, \eqref{item2} becomes
\begin{align*}
(\id \otimes \Delta) (\Delta(P))_{A,B,C}=(\Delta \otimes \id) (\Delta(P))_{A,B,C}.
\end{align*}
Similarly, the projection of $\Delta(P)$ into the component $\SH_A\otimes \SH_{B\cup C}$ is denoted by $ \Delta(P)_{A, B\cup C}$.
It is clear that we only need to show \eqref{item2} after dropping  the sign $(-1)^{(v_2,  \overline{C} v_1)}$ in  $\Delta$.
Applying $\id \otimes \Delta$ to $ \Delta(P)_{A, B\cup C}$, in $\SH_A\otimes \SH_B \otimes  \SH_C$,  we have
\begin{align*}
(\id \otimes \Delta)( \Delta(P)_{A, B\cup C})
=&(\id \otimes \Delta) \left(\frac{ H_A(z_{B\cup C }) P(z_A\otimes z_{B\cup C })}
{\fac(z_{B\cup C }|z_A)}\right)_{A, B, C}\\
=&\frac{H_{A}(z_B) \otimes H_{A}(z_C)  \otimes H_{B}(z_C)  P(z_A\otimes z_{B }\otimes z_C)}{ \fac(z_{B\cup C }|z_A) \fac( z_C|z_B)}.
\end{align*}
Applying $\Delta\otimes \id$ to $ \Delta(P)_{A\cup B,  C}$, in $\SH_A\otimes \SH_B \otimes  \SH_C$,  we have
\begin{align*}
(\Delta\otimes \id)( \Delta(P)_{A\cup B,  C})
=&(\Delta\otimes \id)\left(
\frac{ H_{A\cup B }(z_{ C }) P(z_{A\cup B}\otimes  z_{C })}
{\fac(z_{ C }|z_{A\cup B})}
\right)_{A, B, C}\\
=&\frac{ H_A(z_{ C })  H_B(z_{ C } )H_A(z_{ B } ) P(z_{A}\otimes z_B\otimes z_{C })}
{\fac(z_{ C }|z_{A\cup B})
\fac( z_{B}|z_A)}.
\end{align*}
Therefore, we have $(\mathrm{id}_{\SH} \otimes \Delta) \circ \Delta = (\Delta \otimes \mathrm{id}_{\SH}) \circ \Delta$,  since 
\[\fac(z_{B\cup C }|z_A) \fac( z_C|z_B)=\fac(z_{ C }|z_{A\cup B})
\fac( z_{B}|z_A).\] 
\end{proof}

\begin{proof}[Proof of Theorem~\ref{thm:coprod} \eqref{item3}]
Again it is suffices to show \eqref{item3} on the generators of $\SH^{\ext}$. 

On $\SH^0$, the equality follows from the fact that $\epsilon( H_{k}(w))=1$, which in turn follows from $H_{k}(w)=1+\frac{1}{\fl(h^{(k)})-\fl(w)}$ and the definition of $\epsilon$.

Now we assume $P\in \SH_v$ is a homogeneous element of $\SH$. We  have
\begin{align*}
(\epsilon \otimes \mathrm{id}_{\SH^{\ext}}) \circ \Delta (P)
=(\epsilon \otimes \mathrm{id}_{\SH^{\ext}}) \circ \Delta (P)_{0, v}
=(\epsilon \otimes \mathrm{id})(H_{0}(x_{[1,v]})\otimes P)=P.
\end{align*}
Similarly, $(\mathrm{id}_{\SH^{\ext}} \otimes \epsilon) \circ \Delta (P)
=P.$ 
This completes the proof. 
\end{proof}

\section{A bialgebra pairing}\label{sec:pairing}
The aim of this section is to define the Drinfeld double of the extended shuffle algebra $\SH^{\ext}$. For this purpose, we introduce a bialgebra pairing on $\SH^{\ext}$.

Recall that for a bialgebra  $(A, \star, \Delta)$  with multiplication $\star$, and coproduct $\Delta$, 
the \textit{Drinfeld double} of the bialgebra $A$ is $DA = A \otimes A^\coop$ as a vector space endowed with a  suitable multiplication. Here $A^\coop$ is $A$ as an algebra but with the opposite comultiplication. If $\dim(A)$ is infinite, in order to define $DA$ as a bialgebra, we need a {\it non-degenerate bialgebra pairing}
\[
( \cdot , \cdot) : A \otimes A \to R, \] i.e., an $R$-bilinear non-degenerate pairing such that \[
(a \star b, c) = (a \otimes b, \Delta(c))\text{ and }(c,a \star b) = ( \Delta(c), a \otimes b)\text{ for all $a, b, c \in A $}. 
\]

For a bialgebra $(A, \star, \Delta)$ together with a non-degenerate bialgebra pairing $( \cdot , \cdot)$, the bialgebra structure of $DA=A^-\otimes A^+$, still denoted by $(\star,\Delta)$, is uniquely determined by the following two properties (see, e.g., \cite[\S~2.4]{X}).
\begin{enumerate}
\item $A^-= A^\coop \otimes 1$ and $A^+ = 1 \otimes A$ are both sub-bialgebras of $DA$. 
\item For any $a,b\in A$, write $a^- = a \otimes 1\in A^-$ and $b^+ = 1 \otimes b\in A^+$. Then
\begin{equation}\label{eq: commut rel}
\sum a^-_1 \star b^+_2 \cdot (a_2, b_1) = 
\sum b^+_1 \star a^-_2 \cdot (b_2, a_1), \,\ \text{for all $a, b \in A$,}
\end{equation}
where we follow Sweedler's notation and write  $\Delta(a^-)=\sum a^{-}_1 \otimes a^{-}_2$, 
$\Delta(b^+)=\sum b^+_1 \otimes b^+_2$. 
\end{enumerate}

On $\SH^{\ext,\sph}$ we will describe a natural bialgebra pairing via residues. This pairing turns out to be non-degenerate only on an ad\`ele version of $\SH^{\ext,\sph}$.
\subsection{Scalar product on  ad\`ele ring}
\label{subsec:paring_on_adele}
In this section we work   in the following setup. Take $R=\C$,  and $\bbG$ an algebraic group. We fix a smooth compactification $\overline{\bbG}$ of $\bbG$ and fix a meromorphic section $\omega$ of $T^*\overline{\bbG}$, such that $\omega$ is translation-invariant in $\bbG$ and nowhere vanishing on $\overline{\bbG}$. 
Denote by $F$ the field of rational functions of $\overline{\bbG}$, $\calO_x$ the analytic stalk of the structure sheaf of $\overline{\bbG}$ at a closed point $x\in \overline{\bbG}$, and $K_x$ the  field of fractions of $\calO_x$. 
Let $\bbA$ be the restricted product $\prod_{x\in \overline{\bbG}}'K_x$, i.e., an element in $\bbA$ is a collection of rational functions $\{r_x\}_{x\in \overline{\bbG}}$, so that $r_x\in\calO_x$ except for finitely many $x$. 
We will call $\bbA$ the ring of repartitions.\footnote{In literature this ring is also called the ring of  pre-ad\`eles, where the completion of this ring is called the ring of ad\`eles. Hence the notation $\bbA$.}
We consider a $\C$-valued scalar product on $\bbA$ 
\[
(u, v):=\sum_{x\in \overline{\bbG}} \Res_{x} ( \omega \cdot u_x \cdot v_{-x}),
\] which is a modification of the one used in \cite[Example~3.4]{D86}. Here $-x$ is the group inverse of $x$ extended to $\overline{\bbG}$.

We consider the following extension of the definition of the ring of repartitions on $\overline{\bbG}^{(n)}$ for any $n\in \bbN$. 
For any $x\in \overline{\bbG}$, we have $(x,\dots, x)\in \overline{\bbG}^n$, the complete local ring at which is denoted by $\calO^n_x$. Let $K^n_x$ be the localization of $\calO^n_x$ at the divisor $\fS_n\cdot (\{x\}\times \overline{\bbG}^{n-1})$. Note that $\omega$ induces an $n$-form 
$\omega^n:=\bigwedge_{i=1}^n\omega_i$ 
on $\overline{\bbG}^n$,  where $\omega_i$ is the pullback of $\omega$ to $\overline{\bbG}^n$ via the $i$-th projection. Note that for any $f\in K^n_x$, the residue of $f\cdot \omega^n$ at $(x,\dots, x)$ is well-defined, denoted by $\Res_x f\omega$. 
For any $n$-form $\eta\in H^{0}(U-D, \Omega^n)$ in a neighborhood of certain divisor $D\subset U$, see \cite[p.650]{GH} for the definition of the residue $\Res(\eta)$. Note that in the definition,  $\eta$ could have a higher order pole along the divisor $D\subset U$. 

\begin{lemma}\label{lem:res_ord_general}
For any $\sigma\in \fS_n$ and $f\in K^n_x$, we have \[\Res_x f\omega=\Res_x\sigma(f)\omega.\]
\end{lemma}
\begin{proof}
Recall that by \cite[p.650]{GH}, there is a real $n$-cycle $\Gamma$ so that $\Res_x f\omega=\int_{\Gamma} f\cdot \omega$. For an element $\sigma\in \fS_n$, let  $\sigma: U\to U$ be the map induced by the $\fS_n$ action on $U$. 
We have 
\[
 \int_{\Gamma} f\cdot \omega=\int_{\sigma(\Gamma)} \sigma^* (f\cdot \omega).\]
The lemma follows from the facts that $\sigma^*(\omega)=\sign(\sigma)\omega$, and $\sigma(\Gamma)=\Gamma$ with orientation differing by $\sign(\sigma)$. 
\end{proof}
\begin{lemma}\label{lem:nondeg_general}
The pairing $(\cdot,\cdot)$ on $K^n_x$, sending $f,g\in K^n_x$ to $\Res_x (f\cdot g)\omega$ is non-degenerate. 
\end{lemma}
\begin{proof}
Let $X=\prod_{i=1}^n x_i$ be the function on $\overline{\bbG}^n$,  where $x_i=p_i^*(\fl)$ is the pullback of the local uniformizer $\fl$ of $\overline{\bbG}$ along the $i$-th projection $p_i: \overline{\bbG}^n \to \overline{\bbG}$. Assume $f\neq0$, then there is an integer $l$ so that  $X^lf\omega$ has pole of order 1 along the divisor of $X$.  For any $g\in K^n_x$, there is some $N$ large enough, such that $X^Ng$ is regular. We may only work with regular $g$. By \cite[p.659]{GH}, 
 if $\Res_xX^lfg\omega=0$ for any $g$, then $X^lf\omega$ is regular along $X$, which contradicts with the assumption that $X^lf\omega$ has a pole of order 1. 
\end{proof}

Let  $\bbA^{n}$ be the restricted product of $K^n_x$ over all $x\in \overline{\bbG}$, and let $\bbA^{(n)}= (\bbA^{n})^{ \fS_n}$ be the $\fS_n$-invariant part. 
Similarly, for any $v\in\bbN^I$, we can define the ring of repartitions $\bbA^{(v)}=\prod_{i\in I} \bbA^{(v^i)}$ of $\overline{\bbG}^{(v)}$. For any $f\in \bbA^{(v)}$ and any $x\in \overline{\bbG}$, the $x$-component of $f$ is denoted by $f_x$.

We define the ad\`ele version of the shuffle algebra
$\SH_{\bbA}=\bigoplus_{v\in \bbN^I}\SH_{\bbA, v}$. For any $v\in \N^I$, $\SH_{\bbA, v}$ is the localization of the ad\`ele ring $\bbA^{(v)}$ on $\overline{\bbG}^{(v)}$, localized at $I_{\fac}$, where $I_{\fac}$ is the ideal of the union of the null-divisors of $\fac(x_A|x_B)$ over all $(A,B)\vdash v$. 
The action of  $\SH^0$ on $\SH$ induces an action of $\SH^0$ on $\SH_{\bbA}$. Therefore, we also have an ad\`ele version of extended shuffle algebra $\SH_{\bbA}^{\ext}=\SH^0\rtimes \SH_{\bbA}$. The bialgebra structure on $\SH^{\ext}$ induces a bialgebra structure on $\SH_{\bbA}^{\ext}$. 
Let $\SH_{\bbA}^{\sph,\ext} \subset \SH_{\bbA}^{\ext}$ be the spherical subalgebra. 
\begin{lemma}
The subalgebra $\SH_{\bbA}^{\sph, \ext}$ of $\SH_{\bbA}^{\ext}$ is a sub-bialgebra. 
\end{lemma}
\begin{proof}
Since $\Delta$ is an algebra homomorphism, it suffices to show that $\Delta(f)\in \SH_{\bbA}^{\sph, \ext}$ for any generator $f\in \SH_{\bbA, e_i}^{\sph, \ext}$ with $i\in I$. The latter follows from the formula of $\Delta$ \eqref{eq:delta f}.
\end{proof}

Let  $|A|!=\prod_{i\in I} |A^{i}|!$ and 
we introduce the notation. 
\begin{align*}
\fac(x_A)
:=&\prod_{i \in I}
\prod_{\{s, t \in A^{i} \mid s\neq t\}}
\frac{\fl(x^{(i)}_s-x^{(i)}_t+t_1+t_2)}{\fl(x^{(i)}_t-x^{(i)}_s)}
\cdot\\
& \cdot\prod_{h\in H}\Big(
\prod_{s\in A^{\out(h)}}
\prod_{t\in A^{\inc(h)}}
\fl(x_t^{ \inc(h)}-x_s^{\out(h)}+ m_h t_1)
\prod_{s\in A^{\inc(h)}}
\prod_{t\in A^{\out(h)}}
\fl(x_t^{\out(h)}-x_s^{\inc(h)}+m_{h^*}t_2)
\Big).\end{align*}

Note that the null divisor of $\fac(x_A)$ coincides with the vanishing locus of $I_{\fac}$. 
For any $f,g\in \SH_{\bbA,v}^{\sph,\ext}$, consider the function $\frac{f\cdot g}{\fac(x_A)}$, which {\it a priori} could have a possible pole along the null divisor of $\fac(x_A)$. However, we have the following  
\begin{lemma}\label{lem:sph_regular}
For any $v\in \bbN^I$ and any $f,g\in \SH_{\bbA,v}^{\sph,\ext}$, the function $\frac{f\cdot g}{\fac(x_A)}$
has no poles along the vanishing locus of $I_{\fac}$.
\end{lemma}
\begin{proof}
This is clear if $t_1\neq0$ and $t_2\neq0$. Without loss of generality, we can assume $t_1=0$. By assumption $f,g\in  \SH_{\bbA,v}^{\sph,\ext}$, we take $f\cdot g$ to be product of elements in $\SH_{\bbA,e_i}^{\sph,\ext}$ for $i\in I$. By the multiplication formula \eqref{shuffle formula} and the formula of  $\fac(x_A)$,  the vanishing order of $f\cdot g$ is at least the vanishing order of $\fac(x_A)$ along $I_{\fac}$. This completes the proof.
\end{proof}

We define a non-degenerate bialgebra pairing on the ad\`ele shuffle algebra $\SH_\bbA^{\sph,\ext}$
\[
( \cdot , \cdot) : \SH_\bbA^{\sph,\ext} \otimes \SH_\bbA^{\sph,\ext} \to \bbC
\]
as follows:
\begin{itemize}
\item
For $f\in \SH_{\bbA,v}^{\sph}$, and $g\in \SH_{\bbA,w}^{\sph}$, we define $(f, g)=0$ if $v\neq w$;
\item
For $h\in \SH^0$, and $f\in \SH_{\bbA}^{\sph}$, we define $(h, f)=0$;
\item
For $H_k(u) \in \SH^0[\![u]\!]$, we define 
$(H_k(u), H_k(w))=\frac{\fac(u|w)}{\fac(w|u)}$ for any $k\in I$. 
\end{itemize}
Note that  in particular for $H_k(u) \in \SH^0[\![u]\!]$ we have $(1,H_k(u))=1$.

For any $i\in I$ and $f, g\in \SH_{\bbA,e_i}$, we follow Drinfeld \cite{D86} and define 
\[
(f, g):=\sum_{x\in \overline{\bbG}} \Res_{x} ( f_x \cdot g_{-x}\cdot\omega  ). 
\]
In general, for $f, g\in \SH_{\bbA,v}^{\sph}$, 
we define the pairing $(f, g)$ to be
\[
(f, g):= \sum_{x\in \overline{\bbG}}\Res_{x } \Big( \frac{f(x_A)\cdot g(-x_A)}{ { |A|!} \fac(x_A)} \omega\Big).
\]

For any $x\in \overline{\bbG}$,  the $x$-component of $f$ and $g$ are regular away from  the divisor $\fS_n\cdot (\{x\}\times \overline{\bbG}^{n-1})$ by definition. The only possible pole of $\frac{f(x_A)\cdot g(-x_A)}{ { |A|!} \fac(x_A)}$ is along the vanishing locus of $I_{\fac}$. Hence, 
Lemma~\ref{lem:sph_regular} implies that the $x$-component of $\frac{f(x_A)\cdot g(-x_A)}{ { |A|!} \fac(x_A)}$ is regular away from the divisor $\fS_n\cdot (\{x\}\times \overline{\bbG}^{n-1})$, hence lies in $K_x^v$. In particular, the residue of $\frac{f(x_A)\cdot g(-x_A)}{ { |A|!} \fac(x_A)} \omega$ is well-defined.

\begin{theorem}
Assume the numbers $m_h,m_{h^*}$ associated to $h\in H$ are as in Remark~\ref{rmk:weights}(2). 
On $\SH_\bbA^{\sph,\ext}$, the pairing $(\cdot,\cdot )$ is a  super-symmetric non-degenerate bialgebra pairing. 
\end{theorem} 

This theorem is proved through the following lemmas. 
\begin{lemma}
The pairing $( \cdot , \cdot)$ on $\SH^{\ext}_{\bbA}$ is super-symmetric, i.e., $(f, g)$ equals $(g,f)$ up to a sign.
\end{lemma}
\begin{proof}
We have 
\begin{align*}
(f, g)&= \sum_{x\in \overline{\bbG}}\Res_{x } \Big( \frac{f(x_A)\cdot g(-x_A)}{ { |A|!} \fac(x_A)} \omega(x_A)\Big)\\
&= \sum_{x\in \overline{\bbG}}\Res_{x } \Big( \frac{f(-x_A)\cdot g(x_A)}{ { |A|!} \fac(x_A)} \omega(-x_A)\Big),
\end{align*}
where the last equality used the fact that $\fac(x_A)=\fac(-x_A)$ under the assumption of Remark~\ref{rmk:weights}(2). 
Note that $\omega(x_A)$ and its pullback $\omega(-x_A)$ under group inverse only differ by a sign. Therefore, the above is equal to $( g, f)$ up to a sign.\end{proof}

\begin{lemma}
The above super-symmetric pairing $( \cdot , \cdot)$ on $\SH^{\ext}_{\bbA}$ is non-degenerate. 
\end{lemma}
\begin{proof}
By symmetry, we need to show that if $(f, g)=0$, for any $g\in \SH^{\ext}_{\bbA}$, then $f=0$. This follows from Lemma~\ref{lem:nondeg_general}.
\end{proof}

\begin{lemma}
The above super-symmetric non-degenerate pairing $( \cdot , \cdot)$ on $\SH^{\ext}_{\bbA}$ has the property  \[(a \star b, c) = (a \otimes b, \Delta(c)), \,\ \text{for any $a, b, c\in \SH^{\ext}_{\bbA}$}.\] 
\end{lemma}
\begin{proof}
Let $v_1, v_2\in \N^I$ be two dimension vectors with $v=v_1+v_2$. 
We need to show  $ (f_1\star f_2, P)=(f_1\otimes f_2, \Delta(P))$, for $f_1\in \SH_{\bbA, v_1}, f_2\in \SH_{\bbA, v_2}, P\in \SH_{\bbA, v}$, since $(\SH_{\bbA, v}, \SH_{\bbA, w})=0$, if $v\neq w$. 
Let $(A_o,B_o)=([1,v_1],[v_1+1,v])$ be the standard element in $\bfP(v_1,v_2)$. 
By definition, it suffices to show that
\begin{equation}\label{eq:proof bialgebra}
\sum_{(A, B)\in\bfP(v_1,v_2)}\Big(f_{1}(x_A) \cdot f_2(x_B) \cdot \fac(x_A|x_B), P\Big)
=\Big(f_1(x_{A_o})\otimes f_2(x_{B_o}),  \frac{ H_{A_o}(x_{B_o}) P(x_{A_o}\otimes x_{B_o})}{\fac(x_{B_o}| x_{A_o})}\Big).
\end{equation}
Using Lemma~\ref{lem:res_ord_general}, the left hand side of \eqref{eq:proof bialgebra} is the same as
\begin{align*}
&\sum_{x\in \overline{\bbG}}\Res_{x}\sum_{\{(A, B)\}}\Big(\frac{f_{1}(x_A) \cdot f_2(x_B) \cdot \fac(x_A|x_B)\cdot P(-x_A, -x_B)}{ { |A\cup B|!} \fac(x_{A\cup B}) }\Big) w_{A\cup B}\\
=&
\sum_{x\in \overline{\bbG}}\Res_{x}\Big(\frac{f_{1}(x_{A_o}) \cdot f_2(x_{B_o}) \cdot \fac(x_{A_o}|x_{B_o})\cdot P(-x_{A_o}, -x_{B_o})}{ |{A_o}|! |{B_o}|! \fac(x_{A_o\cup B_o})}\Big) w_{A_o\cup B_o}
\end{align*}
Using the equality
$ H_{A_o}(x_{B_o})P(x_{A_o}\otimes x_{B_o}) =P(x_{A_o}\otimes x_{B_o}) H_{A_o}(x_{B_o})  \widehat{\Phi}(x_{B_o}|x_{A_o})$, 
the right hand side of \eqref{eq:proof bialgebra} is the same as
\begin{align*}
&\Big(f_1(x_{A_o})\otimes f_2(x_{B_o}),  \frac{ P(x_{A_o}\otimes x_{B_o}) H_{A_o}(x_{B_o}) }{\fac(x_{A_o}| x_{B_o})}\Big)\\
=&\sum_{x\in \overline{\bbG}}\Res_{x} 
\frac{f_{1}(x_{A_o})}{ |{A_o}|!  \fac (x_{A_o}) } \cdot  \frac{f_{2}(x_{B_o})}{ |{B_o}|!  \fac (x_{B_o}) } \cdot \frac{ P(-x_{A_o}, -x_{B_o})}{ \fac(-x_{A_o}| -x_{B_o}) } )
w_{A_o}w_{B_o}\\
=&
\sum_{x\in \overline{\bbG}}\Res_{x}\Big(\frac{f_{1}(x_{A_o}) \cdot f_2(x_{B_o}) \cdot \fac(x_{A_o}|x_{B_o})\cdot P(-x_{A_o}, -x_{B_o}) }{ |{A_o}|! |{B_o}|!  \fac (x_{A_o\cup B_o}) }\Big) w_{{A_o}\cup {B_o}},
\end{align*}
Here we used the fact that $(1, H_k(w))=1$ and $\fac(-x_{A_o}| -x_{B_o})=\fac(x_{B_o}| x_{A_o})$ under assumption of Remark~\ref{rmk:weights}(2). Therefore, the equality \eqref{eq:proof bialgebra} holds. 
This completes the proof. 
\end{proof}

\subsection{Quantization of the Manin triples}
Using the non-degenerate bilinear pairing $(\cdot, \cdot)$, we form the Drinfeld double of the bialgebra $\SH^{\sph,\ext}_{\bbA}$, denoted by $D(\SH^{\sph,\ext}_{\bbA})$. Again following the idea of Drinfeld, we consider a subalgebra in the Drinfeld double $D(\SH^{\sph,\ext}_{\bbA})$.

Assume $\Lambda\subset \bbA$ is an isotropic $\bbC$-subring such that $\bbA\cong F\oplus \Lambda$. 
Let $S\subset \overline{\bbG}$ be a non-empty finite subset. Let $\bbA_S$ be the ring of  repartitions without $x$-component for $x\in S$. Let $F_S$ be the subring 
\[\{a\in F\mid \im(a)\in \bbA_S\hbox{ lies in }\im(\Lambda)\subset \bbA_S.\}\] For each $i\in I$, let $\SH_{e_i,S}\subseteq \SH_{e_i,\bbA}$ be the image of $F_S$ under the isomorphism $\SH_{e_i,\bbA}=\bbA$.

We consider the subalgebra of $D(\SH^{\sph,\ext}_{\bbA})$, generated by $\SH_{e_i,S}$,  $\SH_{e_i,S}^{\coop}$ for $i\in I$, and $\SH^0\otimes(\SH^{0})^\coop$. This subalgebra is denoted by $D(\SH^{\sph,\ext}_{S})$.
\begin{lemma}
The comultiplication $\Delta:D(\SH^{\sph,\ext}_{S})\to D(\SH^{\sph,\ext}_{S})\widehat{\otimes} D(\SH^{\sph,\ext}_{S})$ is well-defined. In particular,  $D(\SH^{\sph,\ext}_{S})$ is a bialgebra.
\end{lemma}
\begin{proof}
As the comultiplication is an algebra homomorphism, it suffices to show that the generators of $D(\SH^{\sph,\ext}_{S})$ are closed under the comultiplication. By \eqref{coprodH} it is clear that $\SH^0$ is closed under $\Delta$; and similar for $(\SH^{0})^\coop$. For $f(x^{(k)})\in \SH_{e_k,S}$, $\Delta(f(x^{(k)}))$ can be calculated using \eqref{eq:delta f}. In particular, the right hand side of \eqref{eq:delta f} is an element in $\SH^0\widehat{\otimes}\SH_{e_k,S}$ assuming $f(x^{(k)})\in \SH_{e_k,S}$. Similarly $\Delta(f(x^{(k)}))$ is in $(\SH_{e_k,S})^{\coop}\widehat{\otimes}(\SH^0)^{\coop}$ if $f(x^{(k)})\in (\SH_{e_k,S})^{\coop}$. Hence, we are done.
\end{proof}
 Let $\SH^{\sph,\ext}_{S}$ be the subalgebra of $\SH^{\sph,\ext}_{\bbA}$ generated by $\SH_{e_i,S}$ and $\SH^0$. Clearly $\SH^{\sph,\ext}_{S}$ is a sub-bialgebra of $D(\SH^{\sph,\ext}_{S})$. Similarly we have a sub-bialgebra $(\SH^{\sph,\ext}_{S})^\coop$. As a vector space, we have \[D(\SH^{\sph,\ext}_{S})\cong \SH^{\sph,\ext}_{S}\otimes (\SH^{\sph,\ext}_{S})^\coop.\]

One example of $D(\SH^{\sph,\ext}_{S})$, the case $\bbG=\bbG_a$, is studied in \S~\ref{sec;Yangian}, where we describe a relation  with the Yangian. From \cite{Gr2}, it is expected that when $\bbG=\bbG_m$, $\omega=\frac{1}{x}dx$ with $x$ being the natural coordinate of $\bbA^1\supset\bbG_m$, and $S\subseteq\bbP^1$ consists of $0,\infty$, the algebra $D(\SH^{\sph,\ext}_{S})$ should be related to the quantum loop algebra. However, the proof of this statement involves calculation of non-standard generators of the quantum loop algebra, we postpone it to later investigation.

\begin{example}\label{rmk:connectiveK}
It is equally interesting to consider examples when $\bbG$ is beyond the additive, multiplicative, and elliptic ones. For example, let $R=\bbQ[\beta]$, and $F(u,v)=u+v-\beta uv$ be a formal group law over $R$. An OCT with this formal group law is  the connective $K$-theory \cite{DL}, denoted by $CK$. Let $\calP^{\sph,\ext}(Q,CK)$ be the spherical extended preprojective CoHA with $A=CK$. Then,  $D(\underline\calP^{\sph,\ext}(Q,CK))$ can be considered as a family of algebras over $\Spec R=\bbA^1$, whose generic fiber is $U_q(L\fg_Q)$, and the special fiber is $Y_\hbar(\fg_Q)$. This recovers a classical theorem of Drinfeld that  the quantum loop algebra degenerates to the Yangian.

There is a hyperbolic formal group law $(F(v,u)=\frac{v+u-\mu_1 uv}{1+\mu_2uv},\bbZ[\mu_1^\pm,\mu_2^\pm])$ studied in \cite{LZ}. This formal group law comes from generic singular locus of the Weierstrass elliptic curve. Generically, this formal group law is multiplicative, and at the special fiber the formal group law is additive. Therefore, let $A$ be the OCT with this formal group law, then $\underline\calP^{\sph,\ext}(Q,A)$ should provide another degeneration of $U_q^{\geq 0}(L\fg_Q)$ to $Y^{\geq0}_q(\fg_Q)$ which is different  than the one constructed by Drinfeld.

When $A=\Omega$ is the algebraic cobordism of \cite{LM}, the formal group law is the universal formal group law over the Lazard ring $\bbL$.
We have $\calP^{\sph,\ext}(Q, \Omega)$ acts  on $^\Omega\calM(w):=\Omega_{G_w}(\mathfrak{M}(w))$, the cobordism of the Nakajima quiver variety $\mathfrak{M}(w)$. The eigenvalues of  $\calP^0(Q, \Omega)$ on $^\Omega\calM(w)$ are given by Chern numbers of smooth projective varieties.
We expect this observation to be of geometric applications.

Another potentially interesting OCT is the Morava $K$-theory, whose formal group law is the Lubin-Tate formal group law. However, as the coefficient ring of the Lubin-Tate formal group law has positive characteristic, we will not study this example in the present paper. 
\end{example}

\section{Yangian as a Drinfeld double}\label{sec;Yangian}
In this section we assume the group $\bbG$ is the additive group $\bbG_a$. We assume $Q$ has no edge-loops, and for each $h\in H$ the numbers $m_h$ and $m_{h^*}$ are as in Remark~\ref{rmk:weights}(2).
In this case, we recover the Yangian using the construction of the Drinfeld double of the cohomological Hall algebra.

\subsection{The Yangian}
\label{symmetricYangian} 
The smooth compactification of $\bbG_a$ is $\PP^1$. Let $\fl=x$ be the natural coordinate function of $\bbG_a$. The meromorphic section $\omega$ can be taken as $\omega=dx$, which has an order $2$ pole at $\infty\in \bbG_a$. 
Let $S=\{\infty\}\subset \PP^1$, and  $\Lambda=(\fm_{\infty} \times \prod_{x\in X\backslash S} \calO_x)\subset \bbA$, where $\fm_\infty$ is the maximal ideal in $\calO_\infty$.  
Then, $F_S\cong \bbC[x]$, the ring of regular functions on $\bbG_a$. Hence, $\SH_S^{\sph,\ext}=\SH^{\sph,\ext}$, and the Drinfeld double $D(\SH_S^{\sph,\ext})$ is $\SH^{\sph,\ext}\otimes (\SH^{\sph,\ext})^{\coop}$ endowed with a suitable bialgebra structure.

Define the \textit{reduced Drinfeld double} $\overline D(\SH^{\ext,\sph}_S)$ (see e.g. \cite[2.4]{X})  to be $ D(\SH^{\ext,\sph}_S)$ with the following additional relation imposed
\[H_k^+(u)=H_k^-(-u), \,\ \text{for any $k\in I$}.\] 

Recall that Yangian for a finite dimensional Lie algebra is the quantization of a Manin triple
\cite[\S~4]{D86}.
Fix a Lie algebra $\fa$($\dim \fa<\infty$) and an invariant scalar product on it. Set
\[
\fp=\fa(u^{-1}), \,\  \fp_1=\fa[u], \,\ \fp_2=u^{-1}\fa[u^{-1}]
\] and define the scalar product $\fp$ by 
\[
(f, g):=\Res_{u=\infty} (f(u), g(-u)) du.
\]
The Manin triple $(\fp, \fp_1, \fp_2)$ defines a Lie bialgebra structure on $\fa[u]$. 
The cocommutator is given by 
\[
a(u) \mapsto [a(u)\otimes 1+ 1\otimes a(u), \frac{t}{u-v} ]
\]
where $t$ is the Casimir element.
Identify $\Spec\bbC[u]$ with $\bbG=\bbG_a$, then $\fp_1$ can be alternatively described as the following sub-bialgebra of $\fa\otimes_\bbC F$
\[
G_{S}=\{a\in \fa\otimes_\bbC F\mid \text{the image of a in $ \fa\otimes_\bbC \bbA_{S}$ belongs to the image of $\fa\otimes_\bbC \Lambda$ in $\fa\otimes_\bbC \bbA_{S}$}\},
\]
where $\bbA_{S}$ is the ring of ad\`ele without $\infty$--components. 
The double Yangian, defined as the Drinfeld double of the Yangian, is a quantization of $\fp$.

There is an explicit Drinfeld-type presentation of the Yangian, which applies to any Kac-Moody Lie algebra.
Let $\fg_Q$ be the symmetric Kac-Moody Lie algebra associated to the quiver $Q$.
The  Cartan matrix of $\fg_Q$ is
$(c_{kl})_{k, l\in I}$, which is a symmetric matrix.
Recall that the Yangian of $\fg_Q$, denoted by  $Y_\hbar(\fg_Q)$, is an associative algebra over $\C[\hbar]$, generated by the variables
\[
x_{k, r}^{\pm}, \xi_{k, r}, (k\in I, r\in \N),
\]
subject to relations described below. 
Take the generating series $\xi_{k}(u), x_k^{\pm}(u)\in Y_{\hbar}(\fg_Q)[\![u^{-1}]\!]$ by
\[
\xi_{k}(u)=1+\hbar \sum_{r\geq 0} \xi_{k, r} u^{-r-1}\,\ \text{and}\,\ 
x_k^{\pm}(u)=\hbar \sum_{r\geq 0} x_{k, r}^{\pm} u^{-r-1}. \]
The following is a complete set of relations defining $Y_\hbar(\fg_Q)$ (see, e.g.,  \cite[\S~3.4]{GTL14}):
\begin{description}
\item[(Y1)] For any $i, j\in I$, and $h, h'\in \fh$
\[
[\xi_{i}(u), \xi_{i}(v)]=0, \,\  [\xi_{i}(u), h]=0,\,\  [h, h']=0
\]
\item[(Y2)] For any $i\in I$, and $h \in \fh$,
 \[[h, x_{i}^{\pm}(u)]=\pm \alpha_{i}(h) x_{i}^{\pm}(u)\]
 \item[(Y3)] For any $i, j\in I$, and $a=\frac{\hbar c_{ij}}{2}$
 \[
 (u-v\mp a) \xi_i(u)x_j^{\pm}(v)
=(u-v\pm a) x_j^{\pm}(v) \xi_i(v) \mp 2a x_{j}^{\pm}(u\mp a)\xi_{i}(u)
 \]
 \item[(Y4)] For any $i, j\in I$, and $a=\frac{\hbar c_{ij}}{2}$
 \[
 (u-v\mp a) x_i^{\pm}(u)x_j^{\pm}(v)
=(u-v\pm a) x_j^{\pm}(v) x_i^{\pm}(u)+ \hbar
\Big( [x_{i,0}^{\pm}, x_j^{\pm}(v)]-[x_i^{\pm}(u), x_{j, 0}^{\pm}]\Big) \]
 \item[(Y5)] For any $i, j\in I$, 
 \[
 (u-v)[x^+_i (u),x^-_j (v)]=-\delta_{ij} \hbar(\xi_i(u)-\xi_i(v))
 \]
 \item[(Y6)] For any $i\neq j\in I$, 
 \[
 \sum_{\sigma\in\fS_{1-c_{ij}}}[x_i^{\pm}(u_{\sigma(1)}),[x_i^{\pm}(u_{\sigma(2)}),[\cdots,[x_i^{\pm}(u_{\sigma (1-c_{ij})}),x_j^{\pm}(v)]\cdots]]]=0. \]
\end{description}

\subsection{Relation with the shuffle algebra}
It is shown in \cite[Theorem D]{YZ1} that there is an algebra epimorphism 
 \[Y^{\geq 0}_\hbar(\fg) \surj \SH^{\sph,\ext}|_{t_1=t_2=\frac{\hbar}{2}}, \]
given by $x_{k,r} \mapsto (x^{(k)})^r\in \SH_{e_k}$ for any $k\in I$ and $r\in \bbN$, and $\xi_{k,r}\mapsto (h^{(k)})^r \in \mathbf{S}^{e_k} \subseteq \SH^0$. 
This morphism is compatible with the action of $\SH^{\sph,\ext}$ on cohomology of quiver varieties \cite[Theorem B]{YZ1} and the action of $Y_\hbar(\fg)$  \cite[Theorem~4]{Va00}. Furthermore, this morphism is an isomorphism when $Q$ is of finite type. 

\begin{prop}\label{prop:Yangian coproduct}
Under this epimorphism, the comultiplication $\Delta$ defined in \S~\ref{sec:coprod} agrees with the Drinfeld coproduct of the Yangian.
\end{prop}
\begin{proof}
This follows from comparing the comultiplication formula in \S~\ref{sec:coprod} and the formula of Drinfeld coproduct of the Yangian \cite[4.5]{GTL14} (see also \cite{Her} for the case of quantum loop algebra). 

Indeed, the Drinfeld coproduct of \cite[4.5]{GTL14}, when  $s=0$, applied to $\xi_k(u)$ yields $\xi_k(u)\otimes\xi_k(u)$, which matches up with \eqref{coprodH}. 
Similarly, apply  \cite[4.5]{GTL14}, again when  $s=0$, to $x^+_k(u)=\frac{\hbar}{u-x^{(k)}}$ yields 
\[x^+_k(u)\otimes1+\oint_C\frac{1}{u-v}H_k(v)\otimes(\frac{\hbar}{v-x^{(k)}})dv\] 
where $C$ is a contour near the solo pole $v=x^{(k)}$ of $x^+_k(v)$.  By Cauchy's integral theorem, $\oint_C\frac{1}{u-v}H_k(v)\otimes(\frac{\hbar}{v-x^{(k)}})dv=H_k(x^{(k)})\otimes(\frac{\hbar}{u-x^{(k)}})$. Hence the Drinfeld coproduct applied to $x^+_k(u)$ coincides  with  \eqref{eq:delta f}.
\end{proof}
\begin{remark}
It is well-known that the Drinfeld coproduct on Yangian corresponds to a meromorphic tensor structure whose poles can not be eliminated. On the level of shuffle algebra, this fact is reflected in the poles of  the function $H$ in the right hand side of formula  \eqref{eq:coprod}. Passing to the completion as in \eqref{eqn:completion}, this is reflected by the existence of negative powers as in Example~\ref{ex:completion} which are not well-defined elements in the Yangian.
\end{remark}
In \cite[Theorem~3.20]{Nak12}, the tensor structure on Yangian representations, in terms of the cohomology of quiver varieties, associated to Drinfeld comultiplication has been studied. In other words, the map from $\overline D(\SH^{\sph,\ext})$ to the convolution algebra of the Steinberg variety \cite[\S 5.3]{YZ1} is compatible with the Drinfeld comultiplications on both sides.

Let $f(x^{(k)})$ and $g(x^{(k)})$ be any elements in $\SH_{e_k}$. We 
denote $\sum_{i\in \bbN}(h^{(k)})^i((x^{(k)})^{-i-1}\cdot f,g)$ by $H_k(x^{(k)})(f,g)$ 
then 
we have
\[
H_k(x^{(k)})(f(x^{(k)}),g(x^{(k)}))=\Res_{x^{(k)}=\infty}(H_k(x^{(k)})\cdot f(x^{(k)})\cdot g(-x^{(k)})\cdot dx^{(k)}).\]

Let $E_k(u):=\hbar\sum_{r\geq 0} (x^{(k)})^{r} u^{-r-1} \in \SH_{e_k}[\![u^{-1}]\!]$ be the standard generating series. Note that it is the expansion of the rational function $\frac{\hbar}{u-x^{(k)}}$ on $\bbG_\alpha\times\bbG$ around $u=\infty$. We denote the element $-E_k(-u)^-\in \SH_{e_k}^\coop [\![u^{-1}]\!] \subseteq D(\SH^{\sph,\ext})[\![u^{-1}]\!]$  by $F_k(u)$. 
As a consequence of the epimorphism $Y_\hbar^{\geq 0}(\fg)\surj \SH^{\sph,\ext}$,  we also have  an epimorphism \[Y^{\leq0}_\hbar(\fg)\surj (\SH^{\sph,\ext})^{\coop} \] given by $x^-_k(u)\mapsto F_k(u)$, $\xi_k(u)\mapsto H^-_k(-u)$.

Recall the \textit{reduced Drinfeld double} $\overline D(\SH^{\ext,\sph})$ (see e.g. \cite[2.4]{X})  is $ D(\SH^{\ext,\sph})$ with the following additional relation imposed
\[H_k^+(u)=H_k^-(-u), \,\ \text{for any $k\in I$}.\] 
Set $H_k(u):=-\hbar H_k^+(u)-\hbar H_k^-(-u)$. 
\begin{theorem}\label{thm: rel DSH}
The following relations hold in the reduced Drinfeld double $\overline  D(\SH^{\ext,\sph})$:  
\begin{align}
E_k(u)F_l(v)&=F_l(v)E_k(u), \,\ \text{for $k\neq l$.}\\
E_k(u)F_k(v)&-F_k(v)E_k(u)=-\hbar \left(\frac{H_k(u)-H_k(v)}{u-v} \right).
\end{align}

\end{theorem}

\begin{proof}
First we consider the case when $k=l$. In $\SH^{\ext}$, we have $\Delta(E_k(u))=H_k(x^{(k)})\otimes E_k(u)+E_k(u)\otimes 1$ by \eqref{eq:delta f}. 
We use the relation \eqref{eq: commut rel} with $a=-E_k(-u)$, $b=E_k(v)$ and the fact
\[
(E_k(u), H_l(x))=0, \,\ (1, E_k(u))=0. 
\] It gives the following relation in $\SH^{\ext, \coop}\otimes \SH^{\ext}$: 
\begin{align*}
&H_k^-(x^{(k)})\star1 (-E_k(-u), E_k(v))+ F_k(u)\star E_k(v) (1, H_k(x^{(k)}))\\
=& E_k(v)\star F_k(u) (1, H_k(x^{(k)}))+H_k^+(x^{(k)})\star1 (E_k(v), -E_k(-u)). 
\end{align*}
Therefore, we have
\begin{align*}
[F_k(u), E_k(v)]
&=H_k^+(x^{(k)}) (E_k(v), -E_k(-u))
-H_k^-(x^{(k)}) (-E_k(-u), E_k(v))\\
&=H_k^+(x^{(k)}) \left(\frac{\hbar}{v-x^{(k)}}, \frac{\hbar}{u+x^{(k)}} \right)
-H_k^-(x^{(k)})  \left(\frac{\hbar}{u+x^{(k)}}, \frac{\hbar}{v-x^{(k)}} \right)\\
&=\hbar^2\Res_{x^{(k)}=\infty} 
\frac{ H_k^+(x^{(k)}) d x^{(k)}}{(v-x^{(k)})(u-x^{(k)} )} -
\hbar^2\Res_{x^{(k)}=\infty} 
 \frac{ H_k^-(x^{(k)}) d x^{(k)}}{(u+x^{(k)})(v+x^{(k)} )}\\
&=\hbar^2 \Big(-\frac{H_k^+(u)-H_k^+(v) }{ u-v}
+\frac{H_k^-( -v)-H_k^-( -u)}{-v+u}\Big)\\
&
=\hbar \frac{H_k(u)-H_k(v) }{ u-v}.
\end{align*}

The case when $k\neq l$ is clear, which follows from the relation \eqref{eq: commut rel} with $a=E_k(u)$, $b=E_l(v)$, and $(E_k(u), E_l(v))=0$. 
This completes the proof. 
\end{proof}

The following is a direct corollary to Theorem~\ref{thm: rel DSH}.
\begin{corollary}\label{cor: epi}
Assume $Q$ has no edge-loops.  
We have an algebra epimorphism from 
 $Y_\hbar(\fg)$ to $\overline  D(\SH^{\ext,\sph})$. This map is an isomorphism when $Q$ is of finite type. 
\end{corollary}

\begin{proof}
Construct a map from the Yangian $Y_\hbar(\fg)$ to $\overline D(\SH^{\ext,\sph})$, by
\[
x_k^+(u) \mapsto E_k(u)\,\ , x_k^-(u) \mapsto F_k(u)\,\  \text{and }\,\
\xi_{k}(u)\mapsto H_k(u). 
\] 
By Theorem \ref{thm: rel DSH}, the map respects the relations of the Yangian  $Y_\hbar(\fg)$, hence is an algebra homomorphism. Be definition of $\overline D(\SH^{\ext,\sph})$,  this  map is  surjective. When $Q$ is of finite type,  this map also preserves the triangular decomposition $Y_\hbar(\fg)=Y^+_\hbar(\fg)\otimes Y^0\otimes Y^-_\hbar(\fg)$ and $\overline D(\SH^{\ext,\sph})=\SH^{\sph,\ext}\otimes\SH^0\otimes(\SH^{\sph,\ext})^\coop$. 
According to  \cite[Theorem~D]{YZ1}, this map restricts to an isomorphism on each tensor-factor. Therefore, it is an isomorphism for finite type quiver $Q$. 
\end{proof}
The morphism in Corollary~\ref{cor: epi} is expected to be an isomorphism for a more general class of quivers, which includes the affine Dynkin quivers. We will investigate this in a future publication \cite{GYZ}, based on the results of the present paper.

Proposition \ref{prop:Yangian coproduct} and the corollary above imply Theorem~\ref{introthm:Yangian}, in view of  Proposition~\ref{prop:Shuf_CoHA}.
\subsection{The double Yangian}
The double Yangian $DY_\hbar(\fg)$, which is the Drinfeld double of the Yangian $Y_\hbar(\fg)$, is also of interests. (See, e.g., \cite{KT}.) We expect that $DY_\hbar(\fg)$ can also be realized as a subalgebra in the reduced version of $D(\SH^{\sph,\ext}_{\bbA})$. The construction is similar to that of $D(\SH^{\sph,\ext}_S)$. 
\begin{remark}\label{rmk:doubleYang}
For each $i\in I$, let $\SH_{e_i,K}\subseteq \SH_{e_i,\bbA}$ be the image of $K_\infty$ under the isomorphism $\SH_{e_i,\bbA}=\bbA$.
We consider the subalgebra of $D(\SH^{\sph,\ext}_{\bbA})$, generated by $\SH_{e_i,K}$,  $\SH_{e_i,K}^{\coop}$ for $i\in I$, and $\SH^0\otimes(\SH^{0})^\coop$. This subalgebra, which is closed under comultiplication, is denoted by $D(\SH^{\sph,\ext}_{K})$.
Again, define the \textit{reduced Drinfeld double} $\overline D(\SH^{\ext,\sph}_K)$ to be $ D(\SH^{\ext,\sph}_K)$ with the following additional relation imposed
\[H_k^+(u)=H_k^-(-u), \,\ \text{for any $k\in I$}.\] 
Clearly $\overline D(\SH^{\ext,\sph}_{S}) \subseteq \overline D(\SH^{\sph,\ext}_{K})$ is a sub-bialgebra. 
Taking classical limit, we have an isomorphism \[\overline D(\SH^{\ext,\sph}_K)|_{t_1=t_2=0}\cong U(\fg(u^{-1})).\]
Although  showing the isomorphism between 
 $\overline D(\SH^{\ext,\sph}_{K})$ and the algebra $ DY_\hbar(\fg)$ constructed in \cite{KT} is beyond the scope of the present paper, based on the observation above this is conceivable. 
 \end{remark}
 
\newcommand{\arxiv}[1]
{\texttt{\href{http://arxiv.org/abs/#1}{arXiv:#1}}}
\newcommand{\doi}[1]
{\texttt{\href{http://dx.doi.org/#1}{doi:#1}}}
\renewcommand{\MR}[1]
{\href{http://www.ams.org/mathscinet-getitem?mr=#1}{MR#1}}

\end{document}